\documentclass[a4paper,reqno]{amsart}

\usepackage[spanish,english]{babel}
\usepackage{amsmath,amssymb,amsthm,color}

\newtheorem{Th}{Theorem}[section]

\newtheorem{Rm}{Remark}[section]

\newtheorem{Prop}{Proposition}[section]
\newtheorem{Lem}{Lemma}[section]
\newtheorem{Def}{Definition}[section]

\numberwithin{equation}{section}

\title[$T1$ criterion for Hermite]{A $T1$  criterion for Hermite-Calder\'on-Zygmund operators on the $BMO_H(\mathbb R^n)$ space and applications}

\author[J.J. Betancor]{Jorge J. Betancor}
\address{Departamento de An\'alisis Matem\'atico\\
Universidad de la Laguna\\
Campus de Anchieta, Avda. Astrof{\'\i}sico Francisco S\'anchez, s/n\\
38271 La Laguna (Sta. Cruz de Tenerife), Spain}
\email{jbetanco@ull.es}

\author[R. Crescimbeni]{Raquel Crescimbeni}
\address{Departamento de Matem\'aticas\\
Facultad de Econom\'{\i}a y Administraci\'on\\
Universidad de Co\-mahue\\
8300 Neuqu\'en, Argentina}
\email{rcrescim@uncoma.edu.ar}

\author[J.C. Fari\~na]{Juan C. Fari\~na}
\address{Departamento de An\'alisis Matem\'atico\\
Universidad de la Laguna\\
Campus de Anchieta, Avda. Astrof{\'\i}sico Francisco S\'anchez, s/n\\
38271 La Laguna (Sta. Cruz de Tenerife), Spain}
\email{jcfarina@ull.es}

\author[P.R. Stinga]{Pablo Ra\'ul Stinga}
\address{Departamento de Matem\'aticas\\
Universidad Aut\'onoma de Madrid\\
28049 Madrid, Spain}
\email{pablo.stinga@uam.es}

\author[J.L. Torrea]{Jos\'e L. Torrea}
\address{Departamento de Matem\'aticas
and ICMAT-CSIC-UAM-UCM-UC3M\\
Universidad Aut\'onoma de Madrid\\
28049 Madrid, Spain}
\email{joseluis.torrea@uam.es}

\thanks{Research partially supported by MTM2007/65609,  MTM2008-06621-C02-01 and PCI 2006-A7-0670 from Mi\-nis\-terio de Ciencia e Innovaci\'on, Spain}

\begin{document}

\begin{abstract}
In this paper we establish a $T1$ criterion for the boundedness of Hermite-Calder\'on -Zygmund operators on the $BMO_H(\mathbb{R}^n)$ space naturally associated to the Hermite operator $H$. We apply this criterion in a systematic way to prove the boundedness on $BMO_H(\mathbb R^n)$ of certain harmonic analysis operators related to $H$ (Riesz transforms, maximal operators, Littlewood-Paley $g$-functions and variation operators).
\end{abstract}

\maketitle

\section{Introduction}

It is well-known the crucial role played by $T1$ and its relation with the classical $BMO$ space of John and Nirenberg in the analysis of $L^p$-boundedness of Calder\'on-Zygmund operators $T$ (see \cite{DJ,H,H2} and \cite[p.~590]{Gra}).

Moreover, $T1$ is an important object to understand the behavior of certain classes of integral operators in H\"older spaces. Indeed, in \cite{StinTo} some operators related to the harmonic oscillator (also known as Hermite operator)
\begin{equation}\label{oscilador armonico}
H=-\Delta+|x|^2,\qquad\hbox{in}~\mathbb{R}^n,
\end{equation}
such as the fractional harmonic oscillator $H^\sigma$, the Hermite-Riesz transforms, the fractional integrals $H^{-\sigma}$, among others, are studied when they act on certain H\"older spaces $C^{k,\alpha}_H(\mathbb{R}^n)$, $k\in\mathbb{N}$, $0<\alpha<1$, adapted to $H$. Roughly speaking, these operators $T$ can be expressed as
\begin{equation}\label{operador T}
Tf(x)= \int_{\mathbb R^n} K(x,y)(f(y)-f(x))~dy + f(x)T1(x).
\end{equation}
Here the kernel $K(x,y)$ has a singularity for $x\sim y$, so some regularity is required on $f$ for the integral to be well defined. Looking at how the operator $T$ is written, it is natural to expect that $T1$ is a bounded pointwise multiplier in the class where $f$ belongs to. This is in fact the situation in \cite{StinTo}. Nevertheless, the boundedness of operators like \eqref{operador T} for the case $\alpha =0$ is not covered in \cite{StinTo} (it does not make sense to take 0 as a H\"older exponent). However, since the H\"older spaces $C^\alpha$ can be seen as spaces of $BMO_\alpha$-type (see for instance \cite{Torchinsky}), it would be natural to work with $BMO_H(\mathbb{R}^n)$. Note that $BMO_H(\mathbb{R}^n)$ is the natural substitute as extremal space in the Harmonic Analysis for the Hermite function expansion setting (see Section \ref{Section:BMO}). The last question motivates a characterization of pointwise multipliers on $BMO_H(\mathbb{R}^n)$. We believe that such a result belongs to the folklore, but for completeness we present it here with a proof, see Proposition \ref{multiplicador}. Let us point out that the characterization of pointwise multipliers for the $BMO$ space on the torus (compact support case) was proved by S. Janson \cite{Janson} and for the Euclidean $BMO(\mathbb R^n)$ by E. Nakai and K. Yabuta \cite{Nakai-Yabuta}.

To obtain the boundedness on $BMO_H(\mathbb R^n)$ for operators $T$ of the form \eqref{operador T} it seems natural to impose conditions on $T1$. An answer in this direction is provided in our first main result.

\begin{Th}[$T1$-type criterion] \label{criterioT1}
Let $T$ be a Hermite-Calder\'on-Zygmund operator, see Definition \ref{definicion}. Then $T$ is a bounded operator on $BMO_H(\mathbb{R}^n)$ if and only if there exists $C>0$ such that the following two conditions hold:
\begin{enumerate}
\item[(i)] $\displaystyle \frac{1}{|B(x,\gamma(x))|}\int_{B(x,\gamma(x))}|T1(y)|~dy\leq C$, for every $x\in\mathbb{R}^n$, and
\item[(ii)] $\displaystyle \left(1+\log\left(\frac{\gamma(x)}{s}\right)\right)\frac{1}{|B(x,s)|}\int_{B(x,s)}|T1(y)-(T1)_{B(x,s)}|~dy \leq C, $
for every $x\in \mathbb R^n$ and $s>0$ such that $0<s\leq \gamma(x)$, where $\gamma$ is given by
\begin{equation}\label{definicion radio critico}
\gamma(x):=
\left\{\begin{array}{ll}
       \frac{1}{1+|x|},& |x| \geq 1;\\
       \frac{1}{2},& |x| < 1.
\end{array}\right.
\end{equation}
\end{enumerate}
Here, as usual, $T1_{B(x,s)}=\displaystyle \frac{1}{|B(x,s)|}\int_{B(x,s)}T1(y)~dy$.
\end{Th}

\begin{Rm}[Vector-valued setting]\label{util}
Theorem \ref{criterioT1} can also be stated in a vector valued setting. That is, if $Tf$ takes values in a Banach space $X$ then the result holds when we replace the absolute values appearing in hypothesis $\mathrm{(i)}$ and $\mathrm{(ii)}$ by the norm in $X$.
\end{Rm}

\begin{Rm}[How to apply the result]\label{como aplicar}
Assume that $T1$ is a bounded function in $\mathbb{R}^n$. Then $T1$ satisfies the first condition of Theorem \ref{criterioT1}. The second condition is fulfilled whenever there exists $0<\alpha\leq 1$ such that $|T1(x)-T1(y)| \leq C|x-y|^\alpha$, $x,y \in \mathbb R^n$ (for instance, $\mathrm{(ii)}$ holds if $\nabla T1 \in L^\infty(\mathbb R^n)$).
\end{Rm}

We apply Theorem \ref{criterioT1} in a systematic way to prove that several harmonic analysis operators related to $H$ are bounded on $BMO_H(\mathbb R^n)$. The operators are the maximal operators and Littlewood-Paley $g$-functions associated to the heat and Poisson semigroups for $H$ and the Hermite-Riesz transforms (see Section \ref{Section:Aplicaciones}).

\begin{Th}[Harmonic Analysis operators related to $H$]\label{operadores}
The maximal operators and the Littlewood-Paley $g$-functions associated with the heat $\{W_t^H\}_{t>0}$ and Poisson $\{P_t^H\}_{t>0}$ semigroups generated by $H$ and the Hermite-Riesz transforms are bounded from $BMO_H(\mathbb R^n)$ into itself.
\end{Th}

We also consider variation operators. Let $(X, \mathcal{F},\mu)$ be a measure space and $\{T_t\}_{t>0}$ be an uniparametric family of bounded operators in $L^p(X)$ for some $1 \leq p < \infty$, such that $\displaystyle \lim_{t \rightarrow 0^+}T_tf(x)$ exists for a.e. $x \in X$. In the last years many papers devoted their attention to analyze the speed of convergence of the limit above in terms of the boundedness properties of the $\rho$-variation operator $\mathcal{V}_\rho(T_t)$, $\rho > 2$. Such operator is defined by
$$\mathcal{V}_\rho( T_t)(f)(x) =\sup_{t_j\searrow 0}\left(\sum_{j=1}^\infty|T_{t_j}f(x)-T_{t_{j+1}}f(x)|^\rho\right)^{1/\rho},$$
where the supremum is taken over all the sequences of real numbers $\{t_j\}_{j \in \mathbb N}$ that decrease to zero. The uniparametric families we are interested in are: the heat semigroup $\{W_t^H\}_{t>0}$, the Poisson semigroup $\{P_t^H\}_{t>0}$ and the truncated integral operators for the Hermite-Riesz transforms $\{R_H^\varepsilon\}_{\varepsilon >0}$ (see Section \ref{Section:Aplicaciones} for definitions). The $L^p$--theory for the variation operators related to $\{W_t^H\}_{t>0}$, $\{P_t^H\}_{t>0}$ and $\{R_H^\varepsilon\}_{\varepsilon >0}$ was studied in \cite{CMMTV} and \cite{CMTT}.

\begin{Th}[Variation operators]\label{variacion}
Let $\rho >2$. Denote by $\{T_t\}_{t>0}$ any of the uniparametric families of operators $\{W_t^H\}_{t>0}$, $\{P_t^H\}_{t>0}$ or $\{R_H^\varepsilon\}_{\varepsilon >0}$. Then the variation operator $\mathcal {V}_\rho(T_t)$ is bounded from $BMO_H(\mathbb R^n)$ into itself.
\end{Th}

It is a remarkable fact that all the operators related to $H$ listed above can be seen as vector valued singular integral operators. Therefore Remark \ref{util} will be very useful.

Some of the operators were considered by J. Dziuba\'nski et al. \cite{DGMTZ} in the more general setting of Schr\"odinger operators of the form $\mathcal{L}=-\Delta+V$ and the $BMO_{\mathcal{L}}$-spaces associated to them in $\mathbb{R}^n$, when $n\geq3$. In such a context, the potential $V$ belongs to $RH_s$, the reverse H\"older class of exponent $s$, for some $s>n/2$. Since polynomials are in $RH_s$ for all $s>0$, the Hermite case $V=|x|^2$ is included. It was proved in \cite{DGMTZ} that the maximal operators related to the heat and Poisson semigroups and the square function defined by the heat semigroup in the Schr\"odinger context are bounded operators on $BMO_{\mathcal{L}}$. The procedure developed in \cite{DGMTZ} exploits, in each case, the underlying relationship between the operator considered and its corresponding Euclidean counterpart. More recently, B. Bongioanni, E. Harboure and O. Salinas have studied Schr\"odinger-Riesz transforms associated to $\mathcal{L}$ in $BMO^\beta_{\mathcal{L}}$-spaces, $0\leq\beta<1$, in dimension $n\geq3$, see \cite{BHS}. In particular, they showed that if $s>n$ then the Schr\"odinger-Riesz transforms $\mathcal{R}_i$ are bounded on $BMO_{\mathcal{L}}$. When $n/2<s<n$ the operators $\mathcal{R}_i$ fail to be bounded in $L^p$ for all $p>p_0$, where $p_0>1$ depends on $s$, see the seminal paper by Z. Shen \cite{Shen}. This implies that the $\mathcal{R}_i$ are not bounded on $BMO_{\mathcal{L}}$ if $n/2<s<n$. Finally, in \cite{AST} it was proved that the (generalized) square functions defined by the Poisson semigroup related to $\mathcal{L}$ are bounded on $BMO_{\mathcal{L}}$, for $n\geq3$.

Boundedness of Harmonic Analysis operators in the Hermite setting is well-developed. In particular, boundedness results in $L^p$ for the related Poisson integrals, the Hermite-Riesz transforms and the square functions can be found in the book by S. Thangavelu \cite{Th2}, see also \cite{StTo}.

We would like to point out that our method in this Hermite case works for every $n\geq1$. One of the main novelties of this paper is the boundedness in $BMO_H(\mathbb{R}^n)$ of the variation operators, Theorem \ref{variacion}. Finally, and perhaps this is a more important observation, Theorem \ref{criterioT1} allows us to consider all the Harmonic Analysis operators related to $H$ in a unified way. The key ingredient will be the vector-valued approach. Moreover, we believe that in the cases of boundedness of the maximal operators, our proofs are easier and faster than those presented in \cite{DGMTZ}.

The outline of the paper is as follows. We collect in Section 2 the main definitions and properties related to the space $BMO_H(\mathbb R^n)$. In Section 3, together with the definitions of Hermite-Calder\'on-Zygmund operator and $T1$, we present the proof of Theorem \ref{criterioT1} and the characterization of pointwise multipliers, Proposition \ref{multiplicador}. Applications are developed in Section \ref{Section:Aplicaciones} (proofs of Theorems \ref{operadores} and \ref{variacion}).

Throughout this paper $C$ and $c$ will always denote suitable positive constants, not necessarily the same in each occurrence. Without mentioning it, we will repeatedly apply the inequality $r^\mu e^{-r}\leq C_\mu e^{-r/2}$, $\mu\geq0$, $r>0$, and the fact that $\log\frac{1+s}{1-s}\sim s$ for $s\sim0$, and $\log\frac{1+s}{1-s}\sim -\log(1-s)$ for $s\sim1$.

\section{The space $BMO_H(\mathbb R^n)$}\label{Section:BMO}

J. Dziuba\'nski et al. defined in \cite{DGMTZ} the space $BMO_{\mathcal{L}}(\mathbb R^n)$ naturally associated to a Schr\"odinger operator $\mathcal{L}=-\Delta+V$ in $\mathbb{R}^n$, $n\geq3$, where the nonnegative potential $V$ satisfies a reverse H\"older inequality. It turns out that $BMO_{\mathcal{L}}(\mathbb{R}^n)$ is the natural replacement of $L^\infty(\mathbb R^n)$ in this context. In fact, $BMO_{\mathcal{L}}(\mathbb{R}^n)$ is the dual of the Hardy space $H^1_{\mathcal{L}}(\mathbb R^n)$ associated to $\mathcal{L}$ defined by J. Dziuba\'nski and J. Zienkiewicz in \cite{DZ}. For the definition of $BMO_H(\mathbb{R}^n)$ we take the space of \cite{DGMTZ} in the particular case of the harmonic oscillator \eqref{oscilador armonico}, i.e. $V(x)=|x|^2$, and we extend the definition to all $n\geq 1$.

A locally integrable function $f$ in $\mathbb R^n$ belongs to $BMO_H(\mathbb R^n)$ if there exists $C>0$ such that
\begin{enumerate}
\item[(i)] $\displaystyle\frac{1}{|B|}\int_B|f(x)-f_B|~dx\leq C$, for every ball $B$ in $\mathbb R^n$, and
\item[(ii)] $\displaystyle\frac{1}{|B|}\int_B|f(x)|~dx\leq C$, for every  $B=B(x_0,r_0)$, where $x_0 \in \mathbb R^n$ and $ r_0 \geq \gamma(x_0)$.
\end{enumerate}
Here $f_{B}=\displaystyle\frac{1}{|B|}\int_Bf(x)~dx$, for every ball $B$ in $\mathbb R^n$, and the \textit{critical radii} function $\gamma$ is given by \eqref{definicion radio critico}. The norm $\|f\|_{BMO_H(\mathbb R^n)}$ of $f$ is defined by
$$ \|f\|_{BMO_H(\mathbb R^n)}= \inf\{C\geq0:\hbox{ (i) and (ii) above hold}\}.$$
Applying the classical John-Nirenberg inequality it can be seen that if in (i) and (ii) $L^1$-norms are replaced by $L^p$-norms, for $1<p<\infty$, then the space $BMO_H(\mathbb R^n)$ does not change and equivalent norms appear, see \cite[Corollary~3]{DGMTZ}.

It is not hard to check that for every $C>0$ there exists $M>0$ such that if $|x-y|\leq C\gamma(x)$ then $\displaystyle\frac{1}{M}\gamma(x) \leq \gamma(y) \leq M\gamma(x)$.
\vskip 0.3cm
\noindent\textbf{Covering by critical balls.} According to \cite[Lemma~2.3]{DZ} there exists a sequence of points $\{x_k\}_{k=1}^\infty$ in $\mathbb R^n$ so that if $Q_k$ denotes the ball with center $x_k$ and radius $\gamma(x_k)$, $k\in \mathbb N$, then
\begin{enumerate}
\item[(i)] $\bigcup_{k=1}^\infty Q_k= \mathbb R^n$, and
\item[(ii)] there exists $N\in \mathbb N$ such that $\operatornamewithlimits{card}\{j\in \mathbb N: Q_j^{**} \cap Q_k^{**} \neq \emptyset \} \leq N$, for every $k \in \mathbb N$.
\end{enumerate}
For a ball $B$, $B^\ast$ denotes the ball with the same center than $B$ and twice radius.
\vskip 0.3cm
\noindent\textbf{Boundedness criterion.} In order to prove that an operator $S$ defined on $BMO_H(\mathbb R^n)$ is bounded from $BMO_H(\mathbb R^n)$ into itself, it suffices to see that there exists $C>0$ such that, for every $f \in BMO_H(\mathbb R^n)$ and $k \in \mathbb N$,
\begin{enumerate}
\item[($A_k$)] $\displaystyle \frac{1}{|Q_k|}\int_{Q_k}|Sf(x)|~dx \leq C\|f\|_{BMO_H(\mathbb R^n)}$, and
\item[($B_k$)] $\|Sf\|_{BMO(Q_k^*)} \leq C\|f\|_{BMO_H(\mathbb R^n)}$, where $BMO(Q_k^*)$ denotes the usual $BMO$ space on the ball $Q_k^*$,
\end{enumerate}
see \cite[p.~346]{DGMTZ}.
\vskip 0.3cm
In the following lemma we present an example of a function in $BMO_H(\mathbb R^n)$ that will be useful in the sequel.

\begin{Lem}\label{deffF}
There exists a positive constant $C>0$ such that, for every $x_0 \in \mathbb R^n$ and $0<s\leq \gamma(x_0)$, the function $f(x;s,x_0)$ defined by
$$f(x;s,x_0) = \chi_{[0,s]}(|x-x_0|)\log\left(\frac{\gamma(x_0)}{s}\right)+ \chi_{(s,\gamma(x_0)]}(|x-x_0|)\log\left(\frac{\gamma(x_0)}{|x-x_0|}\right),\quad x\in \mathbb R^n,$$
belongs to $BMO_H(\mathbb R^n)$ and $\|f(\cdot;s,x_0)\|_{BMO_H(\mathbb R^n)} \leq C$.
\end{Lem}

\begin{proof}
Recall that the function $h(x)=\log\left(\frac{1}{|x|}\right)\chi_{[0,1]}(|x|)$ is in $BMO(\mathbb R^n)$, see \cite[p.~520]{Gra}. Hence, for every $R>0$, the function $h_R$ given by
$$h_R(x)=h(x/R), \;\; x \in \mathbb R^n,$$
is in $BMO(\mathbb R^n)$ and $\|h_R\|_{BMO(\mathbb R^n)} \leq C$, where $C$ is independent of $R$. Moreover, for every $R,S>0$, the function $h_{R,S}$ defined by
$$h_{R,S}(x) = \min\{S, h(x/R)\},\;\; x \in \mathbb R^n,$$
belongs to $BMO(\mathbb R^n)$ and $\|h_{R,S}\|_{BMO(\mathbb R^n)}\leq C$, where $C$ does not depend on $R$ and $S$. Then, since for every $x_0 \in \mathbb R^n$ and
$0<s\leq \gamma(x_0)$,
$$f(x;s,x_0) = h_{\gamma(x_0),\log\frac{\gamma(x_0)}{s}}(x-x_0),\;\; x\in \mathbb R^n,$$
the function $f(\cdot;s,x_0) \in BMO(\mathbb R^n)$ and $\|f(\cdot;s,x_0)\|_{BMO(\mathbb R^n)}\leq C$, $x_0 \in \mathbb R^n$ and $0<s<\gamma(x_0)$. It only remains to control the means of $f(\cdot;s,x_0)$ on \textit{large} balls. For that let us first note that
 \begin{align*}
    |f|_{B(x_0,\gamma(x_0))} &= \frac{1}{|B(x_0,\gamma(x_0))|}\int_{B(x_0,\gamma(x_0))}f(x;s,x_0)~dx \\
     &\leq \frac{C}{\gamma(x_0)^n}\left[s^n\log\left(\frac{\gamma(x_0)}{s}\right)+\int_{s<|z|<\gamma(x_0)} \log\left(\frac{\gamma(x_0)}{|z|}\right)~dz\right]\leq C,
 \end{align*}
 where $C$ is independent of $s$ and $x_0$. Let $B=B(z_0,r_0)$, $x_0\in\mathbb R^n$ and $r_0\geq\gamma(z_0)$. We can always assume that $B\cap B(x_0,\gamma(x_0))\neq\emptyset$, since the support of $f$ is the closure of the ball $B(x_0,\gamma(x_0))$. Consider first the easier case: when $r_0\geq\gamma(x_0)$. Then we clearly have $|f|_{B}\leq|f|_{B(x_0,\gamma(x_0))}$ and the computation above applies. On the other hand, if $r_0\leq\gamma(x_0)$, we have that $|x_0-z_0|\leq2\gamma(x_0)$ and by the properties of $\gamma$ given above, $\gamma(x_0)\sim\gamma(z_0)$. Using this last fact and the previous observation, we get $\displaystyle |f|_B\leq\frac{|B(x_0,\gamma(x_0))|}{|B(z_0,\gamma(z_0))|}|f|_{B(x_0,\gamma(x_0))}\leq C$. The proof is complete.
\end{proof}

\section{Proof of Theorem \ref{criterioT1} and characterization of pointwise multipliers}\label{Section:Proofs}

\subsection{On the $T1$-criterion: Theorem \ref{criterioT1}}

Before proving Theorem \ref{criterioT1}  we need to precise the definition of the operator $T$ we are considering.

We denote by $L^2_c(\mathbb R^n)$ the set of functions $f\in L^2(\mathbb R^n)$ whose support $\operatorname{supp}(f)$ is a compact subset of $\mathbb R^n$.

\begin{Def}\label{definicion}
Let $T$ be a bounded linear operator on $L^2(\mathbb R^n)$ such that
$$Tf(x) = \int_{\mathbb R^n}K(x,y)f(y)~ dy, \;\; f \in L^2_c(\mathbb R^n) \;\; \mbox{and} \;\; \mbox{a.e.} \;\;x \notin\operatorname{supp}(f).$$
We shall say that $T$ is a Hermite-Calder\'on-Zygmund operator if
\begin{enumerate}
\item[(1)] $\displaystyle |K(x,y)|\leq\frac{C}{|x-y|^n}~e^{-c\left[|x||x-y|+|x-y|^2\right]}$, for all $x,y \in \mathbb R^n$ with $x\neq y$,
\item[(2)] $\displaystyle|K(x,y)-K(x,z)|+|K(y,x)-K(z,x)| \leq C\frac{|y-z|}{|x-y|^{n+1}}$, when $|x-y|>2|y-z|.$
\end{enumerate}
\end{Def}

Note that every Hermite-Calder\'on-Zygmund operator is also a classical Calder\'on-Zygmund operator, see \cite{Gra}. Examples of Hermite-Calder\'on-Zygmund operators are given in Section \ref{Section:Aplicaciones}.
\vskip 0.3cm
\noindent\textbf{Definition of $Tf$ for $f\in BMO_H(\mathbb R^n)$.} Suppose firstly that $f \in L^2(\mathbb R^n)$. For every $R>0$, let $B_R:=B(0,R)$. We can write
$$Tf= T\left(f\chi_{B_R}\right)+ T\left(f\chi_{B_R^c}\right) = T\left(f\chi_{B_R}\right)+ \lim_{n \rightarrow \infty}T\left(f\chi_{B_R^c \cap B_n}\right)$$
where the limit is understood in $L^2(\mathbb R^n)$. This last identity suggests to define the operator $T$ on $BMO_H(\mathbb R^n)$ as follows. Assume that $f\in BMO_H(\mathbb R^n)$ and $R>1$. By using the Hermite-Calder\'on-Zygmund condition (1) for $K$ we get

\begin{align*}
\int_{B_{2R}^c}|K(x,y)||f(y)|~dy &\leq C \sum^\infty_{j=1} \int_{2^jR<|y|\leq 2^{j+1}R}\frac{e^{-c|x-y|^2}}{|x-y|^n}|f(y)|~dy\\
&\leq C \sum^\infty_{j=1} \int_{2^jR<|y|\leq 2^{j+1}R}\frac{1}{|x-y|^{n+1}}|f(y)|~dy \\
&\leq C \sum^\infty_{j=1} \frac{1}{(2^jR)^{n+1}}\int_{|y| \leq 2^{j+1}R} |f(y)|~dy \leq \frac{C}{R}~ \|f\|_{BMO_H(\mathbb R^n)},
\end{align*}
for every $x\in B_R$. Moreover, if $R<S$ we have
\begin{align*}
T\left(f\chi_{B_S}\right)(x) -T\left(f\chi_{B_R}\right)(x) &= T\left(f\chi_{B_S \setminus B_R}\right)(x)= \int_{B_S \setminus B_R} K(x,y)f(y)~ dy\\
& =  \int_{B_R^c} K(x,y)f(y)~ dy - \int_{B_S^c} K(x,y)f(y)~ dy, \;\mbox{a.e.}\; x\in B_R.
\end{align*}
We define
$$
\mathbb Tf(x) = T\left(f\chi_{B_R}\right)(x) + \int_{B_R^c} K(x,y)f(y)~ dy, \;\;\mbox{a.e.}\; x \in B_R ~\mbox{and}~ R>1.
$$
Note that the definition of $\mathbb Tf$ above is consistent in the choice of $R>1$ in the sense that if $S>R>1$ then the definition using $B_S$ coincides almost everywhere in $B_R$ with the one just given.

Let us derive an expression for $\mathbb Tf$ where $\mathbb T1$ appears that will be useful for the proof of our main result. Let $x_0 \in \mathbb R^n$ and $r_0 >0$. For $B=B(x_0,r_0)$ we write
\begin{equation}\label{descomposicion f}
f=(f-f_B)\chi_{B^*}+ (f-f_B)\chi_{(B^*)^c}+f_B=:f_1+f_2+f_3.
\end{equation}
Let us choose $R>0$ such that $B^* \subset B_R$. Using \eqref{descomposicion f} we get
\begin{align}\label{defBMO}
\mathbb Tf(x)&= T\left(f\chi_{B_R}\right)(x) + \int_{B_R^c}  K(x,y)f(y)~ dy \nonumber \\
&= T\left((f-f_B)\chi_{B^*}\right)(x) + T\left((f-f_B)\chi_{B_R\setminus B^*}\right)(x) + f_BT\left(\chi_{B_R}\right)(x) \nonumber \\
&\quad+ \int_{B^c_R} K(x,y) (f(y)-f_B)~dy + f_B\int_{B_R^c}K(x,y) ~dy \nonumber\\
&= T\left((f-f_B)\chi_{B^*}\right)(x) + \int_{(B^*)^c}K(x,y)(f(y)-f_B)~ dy + f_B\mathbb T1(x),
\end{align}
almost everywhere $x\in B^*$.
\vskip 0.3cm
\noindent\textbf{Proof of Theorem \ref{criterioT1}.} First we shall see that conditions (i) and (ii) on $T1$ imply that $T$ is bounded from $BMO_H(\mathbb R^n)$ into itself. In order to do this we will show that there exists $C>0$ such that the properties ($A_k$) and ($B_k$) stated in Section \ref{Section:BMO} hold for every $k \in \mathbb N$ and $f \in BMO_H(\mathbb R^n)$ when the operator $\mathbb T$ is considered.

We start with ($A_k$). According to \eqref{defBMO},
$$\mathbb Tf(x)=  T\left((f-f_{Q_k})\chi_{Q_k^*}\right)(x) + \int_{(Q_k^*)^c}K(x,y)(f(y)-f_{Q_k})~dy+f_{Q_k}\mathbb T1(x),$$
almost everywhere $x \in Q_k$. As $T$ maps $L^2(\mathbb R^n)$ into $L^2(\mathbb R^n)$, by using H\"older's inequality and \cite[Corollary~3]{DGMTZ},
\begin{align*}
\frac{1}{|Q_k|}\int_{Q_k} \left|T\left((f-f_{Q_k})\chi_{Q_k^*}\right)(x)\right|dx &\leq   C \left(\frac{1}{|Q_k|}\int_{Q_k}\left|T\left((f-f_{Q_k})\chi_{Q_k^*}\right)(x)\right|^2dx\right)^{1/2}\\
&\leq C\left(\frac{1}{|Q_k|}\int_{Q_k^*}\left|f(x)-f_{Q_k}\right|^2dx\right)^{1/2} \leq C\|f\|_{BMO_H(\mathbb R^n)}.
\end{align*}
On the other hand, given $x \in Q_k$, by the size condition (1) of the kernel $K$ it can be checked in a standard way, see for instance \cite{Gra},  that
$$\left|\int_{(Q_k^*)^c} K(x,y)\big(f(y)-f_{Q_k}\big)~ dy\right| \leq C\|f\|_{BMO_H(\mathbb R^n)}.$$
Finally, since (i) holds, we have
$$\frac{1}{|Q_k|}\int_{Q_k}\big|f_{Q_k}\mathbb T1(x)\big|~dx = |f_{Q_k}|\frac{1}{|Q_k|}\int_{Q_k}|\mathbb T1(x)|~dx \leq C \|f\|_{BMO_H(\mathbb R^n)}.$$
Hence, we conclude that ($A_k$) holds for $\mathbb T$ with a constant $C>0$ that does not depend on $k$.

Now we have to prove that $\mathbb T$ satisfies ($B_k$) for a certain $C>0$ that it is independent of $k$. Let $B=B(x_0,r_0)\subseteq Q_k^*$, where $x_0 \in \mathbb R^n$ and $r_0 >0$. Note that if $r_0 \geq \gamma (x_0),$  then $\gamma(x_0) \sim \gamma(x_k) \sim r_0,$  hence proceeding as above we will have
$$\frac{1}{|B|}\int_B \left|\mathbb Tf(x)- (\mathbb Tf)_{B}\right|dx \leq \frac{2}{|B|} \int_B |\mathbb Tf(x)|~dx \leq C\|f\|_{BMO_H(\mathbb R^n)},$$
as soon as we have checked that $\displaystyle\frac{1}{|B|}\int_B|T1(x)|~dx\leq C$. In the definition of $T1$ we can write
$$T1(x)=T(\chi_{Q_k^{**}})(x)+\int_{(Q_k^{**})^c}K(x,y)~dy,\qquad x\in Q_k^\ast.$$
Hence, by hypothesis (i) on $T1$, H\"older's inequality and the size condition (1) on the kernel $K$,
\begin{multline*}
    \frac{1}{|B|}\int_{B}|T1(x)|~dx \leq \frac{C}{|Q_k|}\int_{Q_k}|T1(x)|~dx+\frac{C}{|Q_k^{*}|}\int_{Q_k^\ast\setminus Q_k}|T1(x)|~dx \\
     \leq C+\left(\frac{C}{|Q_k^{*}|}\int_{Q_k^\ast}|T(\chi_{Q_k^{**}})(x)|^2~dx\right)^{1/2}+\int_{Q_k^\ast\setminus Q_k}\int_{(Q_k^{**})^c}K(x,y)~dy~dx\leq C.
\end{multline*}
Assume that $r_0 < \gamma(x_0)$. Using \eqref{defBMO} we have that
\begin{align*}
\frac{1}{|B|}\int_B |\mathbb Tf(x) -(\mathbb Tf)_{B}|~dx & \leq \frac{1}{|B|}\int_B\frac{1}{|B|}\int_B |Tf_1(x) - Tf_1(z)|~dz ~dx \\
&\quad + \frac{1}{|B|}\int_B\frac{1}{|B|}\int_B |F_2(x)-F_2(z)|~dz ~dx \\
&\quad+ \frac{1}{|B|}\int_B |\mathbb Tf_3(x) - (\mathbb Tf_3)_{B}|~dx \\
&=: L_1+L_2+L_3,
\end{align*}
where we defined
$$F_2(x)= \int_{(B^*)^c} K(x,y)f_2(y)~dy,\;\; x \in B,$$
and $f=f_1+f_2+f_3$ as in \eqref{descomposicion f}. By H\"older's inequality and the boundedness in $L^2(\mathbb R^n)$ of $T$,
$$L_1\leq \frac{2}{|B|}\int_B |Tf_1(x)|~dx \leq C \left(\frac{1}{|B|} \int_{B^*}|f(x)-f_B|^2~dx\right)^{1/2} \leq C\|f\|_{BMO_H(\mathbb R^n)}.$$
It is well-known, see for instance \cite{Gra}, that the smoothness property (2) of the kernel $K$ implies that
\begin{equation}\label{a sustituir}
|F_2(x) -F_2(z)| \leq C\|f\|_{BMO_H(\mathbb R^n)},\;\; x,z\in  B.
\end{equation}
Therefore,  $L_2 \leq C\|f\|_{BMO_H(\mathbb R^n)}$. Finally, by using the assumption (ii) on $\mathbb T1$ and \cite[Lemma~2]{DGMTZ}, it follows that
\begin{align*}
L_3 &= |f_B| \frac{1}{|B|}\int_B|\mathbb T1(x) - (\mathbb T1)_{B}|~dx\\
&\leq  C\|f\|_{BMO_H(\mathbb R^n)}\left(1+\log\frac{\gamma(x_0)}{r_0}\right)\frac{1}{|B|}\int_B|\mathbb T1(x) - (\mathbb T1)_{B}|~dx \\
&\leq  C\|f\|_{BMO_H(\mathbb R^n)}.
\end{align*}
Hence, we conclude that $\displaystyle\frac{1}{|B|}\int_B|\mathbb Tf(x) - (\mathbb Tf)_{B}|~dx \leq C\|f\|_{BMO_H(\mathbb R^n)}$ for all $B\subset Q_k^\ast$ and ($B_k$) is proved.

Let us now prove the converse statement. Suppose that $T$ is a bounded operator from $BMO_H(\mathbb R^n)$ into itself. Since the function $g(x)=1$, $x\in\mathbb R^n$, belongs to $BMO_H(\mathbb R^n)$, $\mathbb T1$ is in $BMO_H(\mathbb R^n)$. Then property (i) holds and there exists $C>0$ such that, for every ball $B$,
$$\frac{1}{|B|}\int_B|\mathbb T1(y)-(\mathbb T1)_B|~dy\leq C.$$
Let $x_0\in\mathbb R^n$ and $0<s<\gamma(x_0)$. Consider the function $f(\cdot;s,x_0)$ defined in Lemma \ref{deffF}. Following the argument used in the estimate for the term $L_3$ in the proof of the first part of this Theorem and using the fact that $f(\cdot;s,x_0)\in BMO_H(\mathbb R^n)$, we can find a constant $C>0$ that does not depend on $s$ and $x_0$ such that
$$\log\left(\frac{\gamma(x_0)}{s}\right)\frac{1}{|B(x_0,s)|}\int_{B(x_0,s)}|\mathbb T1(y)-(\mathbb T1)_B|~dy\leq C.$$
Then, condition (ii) holds and the proof of Theorem \ref{criterioT1} is complete.\qed

\subsection{Pointwise multipliers in $BMO_H(\mathbb R^n)$}

\begin{Prop}\label{multiplicador}
Let $g$ be a measurable function on $\mathbb R^n$. We denote by $T_g$ the multiplier operator defined by $T_g(f)=fg$. Then $T_g$ is a bounded operator in $BMO_H(\mathbb R^n)$ if and only if
\begin{enumerate}
\item[(i)] $g \in L^\infty(\mathbb R^n)$; and
\item[(ii)] there exists $C>0$ such that
$$\log\left(\frac{\gamma(x)}{s}\right)\frac{1}{|B(x,s)|}\int_{B(x,s)}|g(y)-g_{B(x,s)}|~dy \leq C,$$
for every $x \in \mathbb R^n$ and every ball $B(x,s)$ with radius $0<s\leq\gamma(x)$, where $\gamma$ is given in \eqref{definicion radio critico}.
\end{enumerate}
\end{Prop}

\begin{Rm}
Condition $\mathrm{(ii)}$ in Proposition \ref{multiplicador} is fulfilled, for instance, when there exists $0<\alpha\leq 1$ such that $|g(x)-g(y)| \leq C|x-y|^\alpha$, $x,y \in \mathbb R^n$.
\end{Rm}

\begin{Rm}\label{T1 y multiplicador}
If for some Hermite-Calder\'on-Zygmund operator $T$ we have that $T1$ defines a pointwise multiplier in $BMO_H(\mathbb R^n)$ then the proposition above and Theorem \ref{criterioT1} imply that $T$ is a bounded operator on $BMO_H(\mathbb R^n)$.
\end{Rm}

\begin{proof}[Proof of Proposition \ref{multiplicador}]
If $g$ is a measurable function in $\mathbb R^n$ satisfying the properties (i) and (ii) in Proposition \ref{multiplicador} we can proceed as in the proof of Theorem \ref{criterioT1} to see that $g$ defines a pointwise multiplier in $BMO_H(\mathbb R^n)$ (note that the kernel of the operator $T=T_g$ is zero).

Suppose next that $g$ is a pointwise multiplier in $BMO_H(\mathbb R^n)$. For the function $f(\cdot;s,x_0)$ defined in Lemma \ref{deffF} and any ball $B=B(x_0,s)$ with $0<s<\frac{\gamma(x_0)}{2}$, by using \cite[Lemma~2]{DGMTZ}, we have
\begin{align*}
    \log\left(\frac{\gamma(x_0)}{s}\right)\frac{1}{|B|}\int_B|g(x)|~dx &= \frac{1}{|B|}\int_B|f(x)g(x)|~dx \\
     &\leq \frac{1}{|B|}\int_B|(fg)(x)-(fg)_B|~dx+(fg)_B \\
     &\leq C\|f\|_{BMO_H(\mathbb R^n)}+\log\left(\frac{\gamma(x_0)}{s}\right)\|fg\|_{BMO_H(\mathbb R^n)} \\
     &\leq C\log\left(\frac{\gamma(x_0)}{s}\right)\|f\|_{BMO_H(\mathbb R^n)},
\end{align*}
hence $|g|_B\leq C$ with $C$ independent of $B$. Therefore, $g$ is bounded. On the other hand, if $x_0\in \mathbb R^n$ and $0<s<\gamma(x_0)$ we have that
\begin{align*}
\lefteqn{\log\left(\frac{\gamma(x_0)}{s}\right)\frac{1}{|B(x_0,s)|}\int_{B(x_0,s)}|g(x)-g_{B(x_0,s)}| ~dx}\\
&\qquad= \frac{1}{|B(x_0,s)|}\int_{B(x_0,s)}|g(x)f(x;s,x_0)-(gf(\cdot;s,x_0))_{B(x_0,s)}|~ dx\\
&\qquad\leq \|gf(\cdot;s,x_0)\|_{BMO_H(\mathbb R^n)} \leq C\|f(\cdot;s,x_0)\|_{BMO_H(\mathbb R^n)} \leq C.
\end{align*}
The constants $C>0$ appearing in this proof do not depend on $x_0 \in \mathbb R^n$ and $0< s < \gamma(x_0)$.
\end{proof}

\section{Applications}\label{Section:Aplicaciones}

Let us recall some definitions and properties of the operators related to the harmonic oscillator, see \cite{Th2}.

According to Mehler's formula \cite[p.~2]{Th2} the heat semigroup $\{W_t^H\}_{t>0}$ generated by $-H$ is given, for every $f \in L^2(\mathbb R^n)$, by
\begin{equation}\label{Hermitecalor}
W_t^Hf(x)\equiv e^{-tH}f(x)=\int_{\mathbb R^n} W_t^H(x,y)f(y)~dy,\;\; x \in \mathbb R^n\; \mbox{and}\;\;t>0,
\end{equation}
where
$$W_t^H(x,y) =\left(\frac{e^{-2t}}{\pi(1-e^{-4t})}\right)^{n/2} e^{-\frac{1}{2}\left[\frac{1+e^{-4t}}{1-e^{-4t}}\left(|x|^2+|y|^2\right)
-\frac{4e^{-2t}}{1-e^{-4t}}~x\cdot y\right]},\quad t>0,~x,y\in\mathbb R^n.$$
Applying S. Meda's change of parameters $t=t(s)=\frac{1}{2}\log\frac{1+s}{1-s}$, $0<s<1$, $t>0$, we obtain the following expression of the kernel of $W_{t(s)}^H$:
\begin{equation}\label{Meda}
W_{t(s)}^H(x,y)=\left(\frac{1-s^2}{4\pi s}\right)^{n/2}e^{-\frac{1}{4}\left[s|x+y|^2+\frac{1}{s}|x-y|^2\right]},\;\; x,y \in \mathbb R^n\;\hbox{and}\; s\in(0,1).
\end{equation}
The semigroup $\{W_t^H\}_{t>0}$ is contractive in $L^p(\mathbb R^n)$, $1\leq p\leq \infty$, and selfadjoint in $L^2(\mathbb R^n)$ but it is not Markovian. Moreover, for every $f\in L^p(\mathbb R^n)$, $1\leq p < \infty$, $\displaystyle \lim_{t \rightarrow 0^+}W_t^Hf(x) = f(x)$ in $L^p(\mathbb{R}^n)$ and a.e. $x \in \mathbb R^n$.

The Poisson semigroup associated to $H$ is given by \textit{Bochner's subordination formula}:
\begin{equation}\label{Poisson}
P_t^Hf(x)\equiv e^{-t\sqrt{H}}f(x)=\frac{1}{\Gamma(1/2)}\int_0^\infty e^{-\frac{t^2}{4u}H}f(x)~e^{-u}~\frac{du}{u^{1/2}},\;\; t>0.
\end{equation}

Suppose now that $f \in BMO_H(\mathbb R^n)$. Clearly for every $t \in (0,\infty)$ and $x \in \mathbb R^n$ the integral
$$\int_{\mathbb R^n} W_t(x,y) f(y) ~dy$$
is absolutely convergent. We define $W_t^Hf$ and $P_t^Hf$, $t>0$, by (\ref{Hermitecalor}) and (\ref{Poisson}) respectively.
\vskip0.2cm
In the following subsections we prove Theorems \ref{operadores} and \ref{variacion}.

\subsection{Maximal operators for the heat and Poisson semigroups associated with the Hermite operator.}\label{Subsection:Maximales}

Our systematic method developed in this paper (Theorem \ref{criterioT1}) allows us to show that the maximal operators $W_*^H$ and $P_\ast^H$, defined by $W_*^Hf=\sup_{t>0}|W_t^Hf|$ and $P_*^Hf=\sup_{t>0}|P_t^Hf|$, are bounded from $BMO_H(\mathbb R^n)$ into itself, for every $n \in \mathbb N$.

The leading idea is to express the operators we are dealing with in such a way that the vector-valued setting can be applied, see Remark \ref{util}. Indeed, it is clear that $W_*^Hf=\|W_t^Hf\|_E$, with $E=L^\infty((0,\infty),dt)$. Hence, to see that $W_*^H$ maps $BMO_H(\mathbb R^n)$ into itself it is enough to show that
$$\hbox{the operator}~V(f): = (W_t^Hf)_{t>0}~\hbox{is bounded from}~BMO_H(\mathbb R^n)~\hbox{into}~BMO_H(\mathbb R^n;E).$$
Here $BMO_H(\mathbb R^n;E)$ is defined in the obvious way by replacing the absolute values $|\cdot|$ by norms $\|\cdot\|_E$. It is well-known that $V$ is bounded from $L^2(\mathbb R^n)$ into $L^2(\mathbb R^n;E)$, see \cite{StTo}. The desired boundedness result can be deduced from Remark \ref{como aplicar} and the following

\begin{Prop}\label{maximal}
There exist positive constants $C$ and $c$ such that
\begin{enumerate}
\item[(i)] $\displaystyle\|W_t^H(x,y)\|_E \leq \frac{C}{|x-y|^n}~e^{-c\left[|x-y|^2+|x||x-y|\right]}$, $x,y\in \mathbb R^n$, $x \neq y$;\vspace{3mm}
\item[(ii)] $\displaystyle\left\|\nabla_xW_t^H(x,y)\right\|_E \leq \frac{C}{|x-y|^{n+1}}$, $x,y \in \mathbb R^n$, $x \neq y$.\vspace{3mm}
\item[(iii)] Moreover, $\displaystyle\|W_t^H1\|_E \in L^\infty(\mathbb R^n)$ and $\displaystyle\left\|\nabla W_t^H1\right\|_E \in L^\infty(\mathbb R^n)$.
\end{enumerate}
\end{Prop}

\begin{proof}
(i) Observe that if $x,y \in \mathbb R^n$, $x\cdot y > 0$, then $|x+y|\geq|y|$ and for all $s \in(0,1)$,
\begin{align}\label{acosem1}
e^{-\frac{1}{4}\left[s|x+y|^2+\frac{1}{s}|x-y|^2\right]} &\leq e^{-\frac{1}{8s}|x-y|^2}e^{-\frac{1}{8}\left[s|x+y|^2+\frac{1}{s}|x-y|^2\right]} \nonumber\\
&\leq  e^{-\frac{1}{8s}|x-y|^2}e^{-\frac{1}{8}|x-y||x+y|} \leq e^{-\frac{1}{8}\left[\frac{1}{s}|x-y|^2+|y||x-y|\right]}.
\end{align}
On the other hand, if $x,y \in \mathbb R^n$, $x\cdot y \leq 0$, then $|x-y|\geq|y|$ and for all $s \in (0,1)$
\begin{equation}\label {acosem2}
e^{-\frac{1}{4}\left[s|x+y|^2+\frac{1}{s}|x-y|^2\right]} \leq e^{-\frac{1}{4s}|x-y|^2}\leq e^{-\frac{1}{8s}|x-y|^2}e^{-\frac{1}{8s}|y||x-y|} \leq e^{-\frac{1}{8}\left[\frac{1}{s}|x-y|^2+|y||x-y|\right]}.
\end{equation}
Therefore, (i) follows from \eqref{Meda}, \eqref{acosem1} and \eqref{acosem2}.

(ii) By (\ref{Meda}),
\begin{align*}
\left|\nabla_xW_{t(s)}^H(x,y)\right| &\leq \frac{C}{s^{n/2}} \left(\frac{|x-y|}{s}+s|x+y|\right) e^{-\frac{1}{4}\left[s|x+y|^2+\frac{1}{s}|x-y|^2\right]} \\
&\leq \frac{C}{s^{(n+1)/2}}~e^{-\frac{c}{s}|x-y|^2}\leq \frac{C}{|x-y|^{n+1}},\quad x,y\in\mathbb R^n,~x\neq y,~\hbox{and}~s\in(0,1).
\end{align*}

(iii) These properties can be easily deduced from the fact that
\begin{equation}\label{acosem4}
W_{t(s)}^H1(x) = \frac{1}{(4\sqrt{\pi})^{n/2}} \left(\frac{1-s^2}{1+s^2}\right)^{n/2}e^{-\frac{s}{1+s^2}|x|^2}, \;x \in \mathbb R^n \;\mbox{and}\; s\in (0,1).
\end{equation}
Indeed, we clearly have $|W_{t(s)}^H1(x)|\leq C$. Moreover,
$$|\nabla W_{t(s)}^H1(x)|\leq C\frac{s}{1+s^2}~|x|e^{-\frac{s}{1+s^2}|x|^2}\leq C\left(\frac{s}{1+s^2}\right)^{1/2}e^{-\frac{cs}{1+s^2}|x|^2}\leq C,$$
for all $0<s<1$ and $x\in\mathbb R^n$.
\end{proof}

In order to see that the maximal operator associated with the Poisson semigroup $P_*^{H}f=\sup_{t>0}|P_t^Hf|=\|P_t^Hf\|_E$ is bounded from $BMO_H(\mathbb R^n)$ into itself we can proceed using the vector-valued setting and the boundedness for the maximal heat semigroup as follows. Let $f\in BMO_H(\mathbb R^n)$. For any ball $B$ we have that
\begin{align*}
\lefteqn{\frac{1}{|B|}\int_B\left\|P_t^{ H}f(x)- \left(P_t^{ H}f\right)_{B}\right\|_{E}~ dx } \\
&= \frac{1}{|B|}\int_B\left\| \frac{1}{\Gamma(1/2)}\int_0^\infty W_{\frac{t^2}{4u}}^Hf(x)e^{-u} \frac{du}{u^{1/2}} - \frac{1}{|B|}\int_B \frac{1}{\Gamma(1/2)}\int_0^\infty W_{\frac{t^2}{4u}}^Hf(y) e^{-u}\frac{du}{u^{1/2}}~dy \right\|_{E}dx \\
&= \frac{1}{|B|}\int_B\left\| \frac{1}{\Gamma(1/2)}\int_0^\infty W_{\frac{t^2}{4u}}^Hf(x)e^{-u} \frac{du}{u^{1/2}}- \frac{1}{\Gamma(1/2)}\int_0^\infty \frac{1}{|B|}\int_B W_{\frac{t^2}{4u}}^Hf(y) ~dy ~e^{-u} \frac{du}{u^{1/2}}\right\|_{E}dx \\
&\leq  C\int_0^\infty \frac{1}{|B|} \int_B \left\|W^H_{\frac{t^2}{4u}}f(x) - \frac{1}{|B|}\int_B W_{\frac{t^2}{4u}}^Hf(y)dy \right\|_{E}dx~ e^{-u}\frac{du}{u^{1/2}} \\
&\leq  C \| W_\ast^Hf \|_{BMO_H(\mathbb R^n)}\int_0^\infty  e^{-u}\frac{du}{u^{1/2}}\leq  C \|f\|_{BMO_H(\mathbb R^n)}.
\end{align*}
If $B=B(x_0,r_0)$ for $x_0 \in \mathbb R^n$ and $r_0 \geq \gamma(x_0)$ then
\begin{align*}
\frac{1}{|B|}\int_B  \left\|P_t^{ H}f(x)\right\|_{E} dx & \leq  C \int_0^\infty \frac{1}{|B|}\int_B \left\| W_{\frac{t^2}{4u}}^Hf(x)\right\|_{E} dx~ e^{-u}\frac{du}{u^{1/2}}  \\
&\leq C \big\|W_\ast^Hf\big\|_{BMO_H(\mathbb R^n)}\int_0^\infty e^{-u}\frac{du}{u^{1/2}} \leq  C\|f\|_{BMO_H(\mathbb R^n)}.
\end{align*}
Therefore $P^H_\ast$ is bounded from $BMO_H(\mathbb{R}^n)$ into itself.

\subsection{Littlewood-Paley $g$-functions for the heat and Poisson semigroups associated with the Hermite operator}

If $\{T_t\}_{t>0}$ denotes  the heat or Poisson semigroup for the Hermite operator, the Littlewood-Paley $g$-function associated with $\{T_t\}_{t\geq 0}$ is defined by
$$g_{T_t}f(x) = \left(\int_0^\infty \left|t\frac{\partial}{\partial t}T_tf(x)\right|^2 \frac{dt}{t}\right)^{1/2}.$$
It is clear that
$$g_{T_t}f =\|t\frac{\partial}{\partial t}T_tf\|_F,\qquad\hbox{where}~F:=L^2\big((0,\infty),\frac{dt}{t}\big).$$
In \cite{StTo2} it was established that $g_{W_t^H}$ defines a bounded operator from $L^2(\mathbb R^n)$ into itself, or, in other words, the operator $U_{W_t^H}$ defined by $\displaystyle U_{W_t^H}f(x,t):= t\frac{\partial}{\partial t} W_t^Hf(x)$ is bounded from $L^2(\mathbb R^n)$ into $L^2(\mathbb R^n;F)$. We denote by
$$ K^H(x,y)= \left( t \frac{\partial}{\partial t}W_t^H(x,y)\right)_{t>0},\qquad x,y \in\mathbb R^n.$$
In order to show that $g_{W_t^H}$ is bounded from $BMO_H(\mathbb R^n)$ into itself it suffices to prove the following estimates and then apply Theorem \ref{criterioT1}.

\begin{Prop}\label{gfuncion}
There exist positive constants $C$ and $c$ such that
\begin{enumerate}
\item[(i)] $\displaystyle\|K^H(x,y)\|_{F} \leq \frac{C}{|x-y|^n}~e^{-c\left[|x-y|^2+|y||x-y|\right]}$, $x,y\in \mathbb R^n$, $x \neq y$;\vspace{3mm}
\item[(ii)] $\displaystyle\left\|\nabla_xK^H(x,y)\right\|_{F} \leq \frac{C}{|x-y|^{n+1}}$, $x,y \in \mathbb R^n$, $x \neq y$.\vspace{3mm}
\item[(iii)] Moreover, $\displaystyle \|U_{W_t^H}1\|_F  \in L^\infty(\mathbb R^n)$ and $\displaystyle\|\nabla U_{W_t^H}1\|_F \in  L^\infty(\mathbb R^n)$.
\end{enumerate}
\end{Prop}
\begin{proof}
(i) Let us first note that, by using \eqref{Meda}, \eqref{acosem1} and \eqref{acosem2},
\begin{align}\label{acosem3}
\left|\frac{\partial W^H_{t(s)}}{\partial s} (x,y)\right| &\leq  C\left( \frac{1}{1-s} + \frac{1}{s}+ |x+y|^2+ \frac{|x-y|^2}{s^2} \right) \left(\frac{1-s}{s}\right)^{n/2} e^{-\frac{1}{4}\left[s|x+y|^2+\frac{1}{s}|x-y|^2\right]} \nonumber\\
&\leq  C \left(\frac{1}{1-s} + \frac{1}{s}\right)\left(\frac{1-s}{s}\right)^{n/2} e^{-\frac{1}{8}\left[s|x+y|^2+\frac{1}{s}|x-y|^2\right]} \nonumber \\
&\leq  C \left[(1-s)^{n/2-1} +s^{-(n/2+1)}\right]e^{-\frac{1}{8}\left[\frac{1}{s}|x-y|^2+|y||x-y|\right]},~x,y\in\mathbb R^n,~s\in(0,1).
\end{align}
Hence, applying Meda's change of parameters $t=t(s)=\frac{1}{2}\log\frac{1+s}{1-s}$ and \eqref{acosem3} we have
\begin{align*}
\|K^H(x,y)\|_F &=\left(\frac{1}{2}\int_0^1\log\left(\frac{1+s}{1-s}\right)\left|\frac{\partial W_{t(s)}^H}{\partial s}(x,y)\right|^2(1-s^2)~ds\right)^{1/2} \\
&\leq  C\left(\int_0^{1/2}\frac{e^{-\frac{|x-y|^2}{8s}}}{s^n}~\frac{ds}{s} - \int_{1/2}^1(1-s)^{n-1}\log(1-s)~ds\right)^{1/2} e^{-\frac{1}{16}\left[|x-y|^2+|y||x-y|\right]}\\
&\leq \frac{C}{|x-y|^n}~e^{-\frac{1}{16}\left[|x-y|^2+|y||x-y|\right]},\;\;x,y\in \mathbb R^n,\;x\neq y.
\end{align*}

(ii) This property was established in \cite[Proposition~2.1]{StTo2}.

(iii) By Meda's change of parameters and \eqref{acosem4},
\begin{align*}
\|U_{W_t^H}1(x)\|_F^2 &\leq C \int_0^1 \log\left(\frac{1+s}{1-s}\right)(1-s)^{n+1}\left(\frac{s}{1-s}+|x|^2\right)^2 e^{-\frac{2s}{1+s^2}|x|^2}~ds \\
&\leq C\int_0^{1/2}\left(s^3+s|x|^4\right)e^{-s|x|^2}~ds - \int_{1/2}^1(1-s)^{n-1}\log(1-s)~ds  \leq C,
\end{align*}
for all $x\in\mathbb R^n$. According to \eqref{acosem4},
$$
\nabla W^H_{t(s)}1(x)= -\frac{1}{(4\pi)^{n/2}}\left(\frac{1-s^2}{1+s^2}\right)^{n/2} \frac{2s x }{1+s^2}~e^{-\frac{s}{1+s^2}|x|^2},\; x\in \mathbb R^n,\;s \in (0,1).
$$
Then
\begin{align*}
\lefteqn{\|\nabla U_{W^H_{t}}1\|_F^2} \\ &\leq C\int_0^1\log\left(\frac{1+s}{1-s}\right)(1-s) \left((1-s)^{n/2-1}s^{3/2}+(1-s)^{n/2+1}|x|\right)^2
 e^{-\frac{s}{c}|x|^2}ds \\ &\leq C,\qquad x \in \mathbb R^n.
\end{align*}
\end{proof}

The boundedness in $BMO_H(\mathbb R^n)$ of $g_{P^{H}_t}$ can be deduced from the $BMO_H(\mathbb R^n)$-boundedness of $g_{W_t^H}$ as in the previous subsection.

\subsection{Hermite-Riesz transforms}\label{Subsection:Riesz}
For every $i=1,2,\ldots,n$, the $i$-th Hermite-Riesz transform $R_i^H$ is defined by
$$R_i^Hf= \frac{\partial}{\partial x_i}H^{-1/2}f,\qquad f\in C_c^\infty(\mathbb R^n).$$
Here $C_c^\infty(\mathbb R^n)$ denotes the space of the $C^\infty$-functions on $\mathbb R^n$ with compact support. The negative square root of the Hermite operator is given by
$$H^{-1/2}f(x) = \frac{1}{\Gamma(1/2)}\int_0^\infty W_t^Hf(x) ~\frac{dt}{t^{1/2}}.$$
The operators $R_i^H$ are bounded from $L^2(\mathbb R^n)$ into itself and, for every $f\in L^2(\mathbb R^n)$,
$$R_i^Hf(x) = \lim_{\varepsilon \rightarrow 0^+}\int_{|x-y| >\varepsilon} R_i^H(x,y)f(y)~ dy, \;\;\mbox{a.e.}\;x\in \mathbb R^n,$$
where
$$R_i^H(x,y) = \frac{1}{\Gamma(1/2)}\int_0^\infty\frac{\partial}{\partial x_i}W_t^H(x,y)~\frac{dt}{t^{1/2}},\;\; x,y\in\mathbb R^n,\;x\neq y,$$
see \cite{Th2} and \cite{StTo}. By proceeding as in the proof of \cite[Lemma~5.6]{StinTo} it can be checked that
\begin{equation}\label{acoriesz1}
\big|R_i^H(x,y)\big| \leq C \frac{e^{-c\left[|x-y|^2+|y||x-y|\right]}}{|x-y|^n},\;x,y\in \mathbb R^n,\;x\neq y,
\end{equation}
and (see also \cite{StTo})
\begin{equation}\label{acoriesz2}
\sum_{j=1}^n\left(\left|\frac{\partial}{\partial x_j}R_i^H(x,y)\right| + \left|\frac{\partial}{\partial y_j}R_i^H(x,y)\right| \right) \leq C\frac{e^{-c\left[|x-y|^2+|y||x-y|\right]}}{|x-y|^{n+1}},\;x,y\in \mathbb R^n,\;x\neq y.
\end{equation}
Hence, the Hermite-Riesz transforms are Hermite-Calder\'on-Zygmund operators. Then, the boundedness in $BMO_H(\mathbb R^n)$ of $R_i^H$ can be deduced from Remark \ref{T1 y multiplicador} and the following

\begin{Prop} \label{ProRiesz}
Let $i=1,\ldots,n$. Then $g:=R_i^H1$ defines a bounded pointwise multiplier in $BMO_H(\mathbb R^n)$.
\end{Prop}

\begin{proof}
We have to check conditions (i) and (ii) of Proposition \ref{multiplicador}. Since
$$R_i^H1 (x) = \lim_{\varepsilon \rightarrow 0^+}\int_{|x-y|>\varepsilon}R_i^H(x,y)~dy, \;\; \mbox{a.e.}\;x \in \mathbb R^n,$$
according to \cite[Lemma 5.10]{StinTo}, (i) holds. Assertion (ii) in Proposition \ref{multiplicador} can be proved by using the procedure developed in the proof of the corresponding property for the variation operator given in the next section. Since that proof is more involved than this one we prefer to put the complete description in the last subsection.
\end{proof}

The proof of Theorem \ref{operadores} is complete.

\vskip0.2cm

The last two subsections of this paper are devoted to prove Theorem \ref{variacion}.

\subsection{$\rho$-variation for the heat and Poisson semigroups associated with the Hermite operator}\label{Subsection:variacion semigrupos}

Here we prove Theorem \ref{variacion} for the semigroups $\{W_t^H\}_{t>0}$ and $\{P^{H}_t\}_{t>0}$.

Let us first analyze the operator $\mathcal{V}_\rho(W_t^H)$, $\rho>2$. If, as above, $t=t(s)=\frac{1}{2}\log\frac{1+s}{1-s}$, $s\in (0,1)$, then
$$\sup_{t_j\searrow 0}\left(\sum_{j=1}^\infty|W^H_{t_j}f(x)-W^H_{t_{j+1}}f(x)|^\rho\right)^{1/\rho} = \sup_{\substack{s_j \searrow 0\\ 0<s_j<1}}\left
(\sum_{j=1}^\infty|W^H_{t(s_j)}f(x)-W^H_{t(s_{j+1})}f(x)|^\rho\right)^{1/\rho}.$$
Therefore, when dealing with $\mathcal{V}_\rho(W_t^H)$ the expression for the kernel $W_{t(s)}^H(x,y)$ given by (\ref{Meda}) can be used. In order to prove our result we apply Theorem \ref{criterioT1} in a vector-valued setting, see Remark \ref{util}. Consider the Banach space $E_\rho$ defined as follows. A complex function $h$ defined in $[0,\infty)$ is in $E_\rho$, $\rho>2$, when
$$\|h\|_{E_\rho} := \sup_{t_j \searrow 0}\left(\sum_{j=1}^\infty |h(t_j)-h(t_{j+1})|^\rho\right)^{1/\rho} < \infty.$$
Clearly,
$$\mathcal{V}_\rho(W_t^H)(f)(x) = \|W_t^Hf(x)\|_{E_\rho}, \;\; x\in \mathbb R^n.$$
It is known that $\mathcal{V}_\rho(W_t^H)$ maps $L^2(\mathbb R^n)$ into itself, see \cite{CMMTV}. In order to prove that $\mathcal{V}_\rho(W_t^H)$ is bounded from $BMO_H(\mathbb R^n)$ into itself it suffices to see that the operator $\mathbb V_\rho$ defined by
$$\mathbb V_\rho(f) = (W^H_tf)_{t>0},\qquad f \in BMO_H(\mathbb R^n),$$
is bounded from $BMO_H(\mathbb R^n)$ into $BMO_H(\mathbb R^n; E_\rho)$. To this end, according to Theorem \ref{criterioT1}, we only have to check that the kernel $(W_t^H(x,y))_{t>0}$ satisfies the properties stated in the following

\begin{Prop}\label{semigrupoHCZ}
Let $\rho >2$. There exist positive constants $C$ and $c$ such that
\begin{enumerate}
\item[(i)] $\displaystyle\|W_t^H(x,y)\|_{E_\rho} \leq \frac{C}{|x-y|^n}~e^{-c\left[|y||x-y|+|x-y|^2\right]}$, $x,y\in \mathbb R^n$, $x \neq y$;\vspace{3mm}
\item[(ii)] $\displaystyle\left\|\nabla_xW_t^H(x,y)\right\|_{E_\rho} \leq \frac{C}{|x-y|^{n+1}}$, $x,y \in \mathbb R^n$, $x \neq y$.\vspace{3mm}
\item[(iii)] Moreover, $\displaystyle\mathbb V_\rho(1) \in L^\infty(\mathbb R^n;E_\rho)$ and $\displaystyle\nabla\mathbb V_\rho(1) \in  L^\infty(\mathbb R^n;E_\rho)$.
\end{enumerate}
\end{Prop}

\begin{proof}
(i) Let $\{s_j\}_{j=1}^\infty \subset (0,1)$ be a decreasing sequence such that $\lim_{j \rightarrow \infty}s_j=0$. By \eqref{acosem3} we have
\begin{align*}
\lefteqn{\left(\sum_{j=1}^\infty|W^H_{t(s_j)}(x,y)-W^H_{t(s_{j+1})}(x,y)|^\rho\right)^{1/\rho}}\\
 &\hspace{5mm}\leq  \sum_{j=1}^\infty|W^H_{t(s_j)}(x,y)-W^H_{t(s_{j+1})}(x,y)| \leq \int_0^1 \left|\frac{\partial}{\partial s}W^H_{t(s)}(x,y)\right| ds \\
&\hspace{5mm}\leq  C e^{-\frac{1}{16}\left[|x-y|^2+|y||x-y|\right]}\int_0^1 \left((1-s)^{n/2-1} +s^{-(n/2+1)}\right) e^{-\frac{1}{16s}|x-y|^2}ds \\
&\hspace{5mm}\leq C e^{-\frac{1}{16}\left[|x-y|^2+|y||x-y|\right]}\left(\int_0^{1/2} \frac{1}{s^{n/2}} ~ e^{-\frac{1}{16s}|x-y|^2}~\frac{ds}{s} + \int_{1/2}^1 (1-s)^{n/2-1}ds \right) \\
&\hspace{5mm}\leq C e^{-\frac{1}{16}\left[|x-y|^2+|y||x-y|\right]}\left(|x-y|^{-n} \int_0^\infty u^{n/2}e^{-u}\frac{du}{u} +1\right) \\
&\hspace{5mm}\leq  \frac{C}{|x-y|^n}~ e^{-\frac{1}{32}\left[|x-y|^2+|y||x-y|\right]},\;\; x,y \in \mathbb R^n, \;x \neq y.
\end{align*}

(ii) This estimate was proved in \cite[p.~90]{CMMTV}.

(iii) By \eqref{acosem4} we have
\begin{align*}
\mathcal{V}_\rho(W_t^H)(1)(x) & \leq  \int_0^1 \left|\frac{\partial}{\partial s} W_{t(s)}^H1(x)\right|ds \\
& \leq  C \int_0^1 \left[(1-s)^{-1}+ |x|^2(1-s)\right](1-s)^{n/2}e^{-\frac{s}{1+s^2}|x|^2}ds \\
&\leq  C \left( \int_0^{1/2} (1+|x|^2)e^{-\frac{s}{2}|x|^2}ds + \int_{1/2}^1 (1-s)^{n/2-1}~ds \right) = C,\;\; x \in \mathbb R^n.
\end{align*}
This proves that $\mathbb V_\rho(1) \in L^\infty(\mathbb R^n;E_\rho)$.
By using again \eqref{acosem4} we get
\begin{align*}
\lefteqn{\sum_{j=1}^\infty \left| \nabla\left(W^H_{t(s_j)}1(x)-W^H_{t(s_{j+1})}1(x)\right)\right|}\\
&\hspace{5mm}\leq C|x| \sum_{j=1}^\infty \left|\frac{s_j}{1+s_j^2}\left(\frac{1-s_j^2}{1+s_j^2}\right)^{n/2}e^{-\frac{s_j}{1+s_j^2}|x|^2} - \frac{s_{j+1}}{1+s^2_{j+1}}\left(\frac{1-s_{j+1}^2}{1+s_{j+1}^2}\right)^{n/2} e^{-\frac{s_{j+1}}{1+s_{j+1}^2}|x|^2}\right| \\
&\hspace{5mm}\leq  C|x| \int_0^1 \left|\frac{\partial}{\partial s}\left[\frac{s}{1+s^2}~W_{t(s)}^H1(x)\right]\right|ds.
\end{align*}
Since
\begin{align*}
\left|\frac{\partial}{\partial s}\left[\frac{s}{1+s^2}~W_{t(s)}^H1(x)\right]\right| & \leq  C\left((1-s)W_{t(s)}^H1(x) + s\frac{\partial}{\partial s}W_{t(s)}^H1(x)\right) \\
&\leq  C \left((1-s)+ \frac{s}{1-s} + s(1-s)|x|^2 \right)(1-s)^{n/2} e^{-\frac{s}{1+s^2}|x|^2} \\
&\leq  C\left(\frac{1}{1-s}+s|x|^2\right)(1-s)^{n/2} e^{-\frac{s}{1+s^2}|x|^2},\;x \in \mathbb R^n\; \mbox{and}\; s\in (0,1),
\end{align*}
it follows that
\begin{align*}
\int_0^1  \left|\frac{\partial}{\partial s}\left[\frac{s}{1+s^2}~W_{t(s)}^H1(x)\right]\right|~ ds
& \leq  C\int_0^{1/2} (1+s|x|^2)~e^{-\frac{s}{c}|x|^2}~ds
+ (1+|x|)~e^{-c|x|^2} \\
&\leq  C \left(\chi_{\{|x| < 1\}}(x) + \frac{1}{|x|^2}\chi_{\{|x| \geq 1\}}(x)\right), \; x \in \mathbb R^n.
\end{align*}
Hence, $ \nabla\mathbb V_\rho(1) \in L^\infty(\mathbb R^n;E_\rho)$.
\end{proof}

To verify that the variation operator associated with the Poisson semigroup $\mathcal{V}_\rho(P_t^{H})$ is bounded from $BMO_H(\mathbb R^n)$ into itself we can proceed as in the final part of Subsection \ref{Subsection:Maximales} by replacing the space $E$ by $E_\rho$. Details are left to the reader.

\subsection{$\rho$-variation of Hermite-Riesz transforms}

In order to simplify the notation and computations we establish the $BMO_H$-boundedness of the $\rho$-variation operator of the Hermite-Riesz transforms in dimension one. The result in higher dimensions can be proved in a similar fashion. The rather cumbersome computations are left to the interested reader.

As it was mentioned in Subsection \ref{Subsection:Riesz}, the Hermite-Riesz transform $R^H$ is a Hermite-Calder\'on-Zygmund operator. For every $\varepsilon >0$ we set
$$R_\varepsilon^Hf(x)= \int_{|x-y|>\varepsilon} R^H(x,y)f(y)~dy, \;\; x\in \mathbb R.$$
To describe the vector-valued setting, consider the Banach space $E_\rho$ given in Subsection \ref{Subsection:variacion semigrupos}. We have
\begin{align*}
\mathcal{V}_\rho(R^H_\varepsilon)f(x)&:= \sup_{\varepsilon_j\searrow 0}\left(\sum_{j=1}^\infty \Big|R^H_{\varepsilon_j}f(x)-R^H_{\varepsilon_{j+1}}f(x)\Big|^\rho\right)^{1/\rho}\\
  &= \Big\|\Big(\int_{|x-y|>\varepsilon}R^H(x,y)f(y)~dy\Big)_{\varepsilon>0}\Big\|_{E_\rho}.
\end{align*}
Now define the operator $U$ by
$$Uf(x)=\Big(\int_{|x-y|>\varepsilon}R^H(x,y)f(y)~dy\Big)_{\varepsilon>0},\,\,\,x\in \mathbb{R}.$$
To prove that $\mathcal{V}_\rho(R^H_\varepsilon)$ is bounded from $BMO_H(\mathbb{R})$ into itself it is enough to show that the operator $U$ given above is bounded from $BMO_H(\mathbb{R})$ into $BMO_H(\mathbb{R}; E_\rho)$. For that we will apply Theorem \ref{criterioT1} in this vector-valued setting. The first thing to check is that $U$ is a (vector-valued) Hermite-Calder\'on-Zygmund operator, see Definition \ref{definicion}. By one hand, the size condition (1) in Definition \ref{definicion} is valid (see \eqref{acoriesz1}). However, there is a problem with  the smoothness condition (2). The problem is due to the fact that the kernel of the operator $U$, namely $\left\{\chi_{|x-y| >\varepsilon}R^H(x,y)\right\}_{\varepsilon>0}$ in $E_\rho$, cannot be differentiated with respect to $x$. If we follow the proof of the ``if'' part of Theorem \ref{criterioT1}, we can see that the smoothness condition (2) in Definition \ref{definicion} for the kernel $K$ is only applied to prove \eqref{a sustituir}. Hence, we must prove estimate \eqref{a sustituir} for $K=(\hbox{kernel of the operator}~U)$ in an alternative way. This is done in Lemma \ref{Lem:geometrico} below. To overcome the difficulty of estimating a non-smooth kernel, the proof of Lemma \ref{Lem:geometrico} uses a geometric argument introduced for the first time in \cite{GiTo}. Finally, to conclude that $U$ is bounded from $BMO_H(\mathbb{R})$ into $BMO_H(\mathbb{R}; E_\rho)$ we check hypothesis (i) and (ii) of Theorem \ref{criterioT1} on $U1$. Again, the difficulty arises when we want to verify (ii) and the geometric argument of \cite{GiTo} will be needed.

\subsubsection{Alternative proof of \eqref{a sustituir}.} It is clear that we have to begin by proving the following

\begin{Lem}\label{Lem:geometrico}
Let $f\in BMO_H(\mathbb{R})$ and $B=B(x_0,r_0)$, for some $r_0>0$. Set
$$U(f_2)(x):= \int_{(B^*)^c}\chi_{|x-y| >\varepsilon}(y)R^H(x,y)\left(f(y)-f_B\right)~dy,\;\; x \in B.$$
Then $\|U(f_2)(x)-U(f_2)(z)\|_{E_\rho}\leq C\|f\|_{BMO_H(\mathbb R^n)}$, for all $x,z\in B$.
\end{Lem}

\begin{proof}
 We shall use a variant of the geometric argument developed for the first time in \cite{GiTo}. Let $x,y\in B$. We can write
\begin{align}
 \|U(f_2)(x)-U(f_2)(y)\|_{E_\rho} &=\left\|\left(\int_{|x-z|>\varepsilon} R^H(x,z)(f(z)-f_B)\chi_{(B^{*})^c}(z)~dz\right)_{\varepsilon>0}\right.\label{sustituido}\\
&\qquad-\left.\left(\int_{|y-z|>\varepsilon}R^H(y,z)(f(z)-f_B) \chi_{(B^{*})^c}(z)~dz\right)_{\varepsilon>0}\right\|_{E_\rho}\nonumber\\
&\le\left\|\left(\int_{|x-z|>\varepsilon}(R^H(x,z)-R^H(y,z))(f(z)-f_B) \chi_{(B^{*})^c}(z)~dz\right)_{\varepsilon>0}\right\|_{E_\rho}\nonumber\\
&\qquad+\sup_{\varepsilon_j\searrow 0}\Big(\sum_{j=1}^\infty \Big|\int_{-\infty}^\infty
(\chi_{\varepsilon_{j+1}<|x-z|<\varepsilon_j}(z)-\chi_{\varepsilon_{j+1}<|y-z|<\varepsilon_j}(z))\nonumber\\
&\qquad\qquad\qquad\qquad\times R^H(y,z)(f(z)-f_B)\chi_{(B^{*})^c}(z)~dz\Big|^\rho\Big)^{1/\rho} \nonumber\\
&=:A_1(x,y)+\sup_{\varepsilon_j\searrow
0}A_{2,\varepsilon_j}(x,y).\nonumber
\end{align}
By Minkowski's inequality and the smoothness of $R^H(x,y)$,
\begin{align*}
A_1(x,y)&\le \int_{-\infty}^\infty
\Big|R^H(x,z)-R^H(y,z)\Big||f(z)-f_B|\chi_{(B^{*})^c}(z)~dz\\
&\le C\sum_{k=0}^\infty \int_{2^{k+2}r_0<|x_0-z|\le
2^{k+3}r_0}\frac{|x-y|}{|x_0-z|^2}|f(z)-f_B|~dz\\
&\le C\sum_{k=0}^\infty \frac{r_0}{(2^kr_0)^2}\int_{|x_0-z|\le
2^{k+3}r_0}|f(z)-f_B|~dz\le C\|f\|_{BMO_H(\mathbb{R})}.
\end{align*}
We now estimate $A_{2,\varepsilon_j}$. Here we need to introduce the geometric argument of \cite{GiTo}. The factor $\chi_{\{\varepsilon_{j+1} < |x -z | < \varepsilon_{j}\}} -\chi_{\{\varepsilon_{j+1} < |y -z | < \varepsilon_{j}\}} $ will be non-zero if either $\chi_{\{\varepsilon_{j+1} < |x -z | <
\varepsilon_{j}\}} =1$  and $\chi_{\{\varepsilon_{j+1} < |y -z | <
\varepsilon_{j}\}}= 0$ or $\chi_{\{\varepsilon_{j+1} < |x -z | <
\varepsilon_{j}\}} =0$  and $\chi_{\{\varepsilon_{j+1} < |y -z | <
\varepsilon_{j}\}}= 1$. This means that the integral in $A_{2,\varepsilon_j}$ will be non-zero in
the following cases
\begin{itemize}
\item  $\varepsilon_{j+1} < |x -z | < \varepsilon_{j} $ and $|y-z| < \varepsilon_{j+1},$
\item   $\varepsilon_{j+1} < |x -z | < \varepsilon_{j} $ and $|y-z| > \varepsilon_{j},$
\item   $\varepsilon_{j+1} < |y -z | < \varepsilon_{j} $ and $|x-z| < \varepsilon_{j+1},$
\item   $\varepsilon_{j+1} < |y -z | < \varepsilon_{j} $ and $|x-z| > \varepsilon_{j}.$
\end{itemize}
In the first case  we observe that,  $\varepsilon_{j+1} < |x -z |
\le |x-y|+ |y- z| < |x-y| + \varepsilon_{j+1}$. Analogously in the
third case we have $\varepsilon_{j+1} < |y -z |    \le |x-y|+ |x- z|
< |x-y| + \varepsilon_{j+1}$.  In the second case we have
$\varepsilon_j < |y-z| \le |y-x|+|x-z| < |x-y| + \varepsilon_j$ and
analogously in the fourth  case we have $\varepsilon_j < |x-z| \le
|x-y|+|y-z| < |x-y| + \varepsilon_j.$ We fix $1<q<\rho$. Therefore, using H\"older's inequality and the continuous inclusion $\ell^1\subset\ell^{\rho/q}$,
\begin{align*}
\lefteqn{A_{2,\varepsilon_j}(x,y)}\\
&\le   \Big[\sum_{j=1}^\infty \Big|\int_{(B^\ast)^c}\chi_{\{\varepsilon_{j+1} < |x -z | < |x-y| + \varepsilon_{j+1}\}}(z)  \chi_{\{\varepsilon_{j+1} < |x -z | < \varepsilon_{j}\}}(z)|R^H(y,z)|  | f(z) - f _B|~  dz\Big|^\rho\Big]^{1/\rho} \\
&\quad +\Big[\sum_{j=1}^\infty \Big|\int_{(B^{*})^c} \chi_{\{\varepsilon_{j} < |y -z |     < |x-y| + \varepsilon_{j}\}}(z)  \chi_{\{\varepsilon_{j+1} < |x -z | < \varepsilon_{j}\}}(z)|R^H(y,z)|  | f(z) - f _B|~  dz\Big|^\rho\Big]^{1/\rho} \\
&\quad+\Big[\sum_{j=1}^\infty \Big|\int_{(B^{*})^c}\chi_{\{\varepsilon_{j+1} < |y -z |     < |x-y| + \varepsilon_{j+1}\}}(z)  \chi_{\{\varepsilon_{j+1} < |y -z | < \varepsilon_{j}\}}(z)|R^H(y,z)|  | f(z) - f _B|~  dz\Big|^\rho\Big]^{1/\rho} \\
&\quad+\Big[\sum_{j=1}^\infty \Big|\int_{(B^{*})^c} \chi_{\{\varepsilon_{j} < |x-z |     < |x-y| + \varepsilon_{j}\}}(z)  \chi_{\{\varepsilon_{j+1} < |y -z | < \varepsilon_{j}\}}(z)|R^H(y,z)|  | f(z) - f _B|~  dz\Big|^\rho\Big]^{1/\rho} \\
&\le
C \Big[\sum_{j=1}^\infty\Big(\int_{(B^{*})^c}\chi_{\{\varepsilon_{j+1}
< |x -z | < \varepsilon_{j}\}}(z)|R^H(y,z)|^q | f(z) - f
_B|^q~dz\Big)^{\rho/q}\Big]^{1/\rho}|x-y|^{1/q'}\\
&\quad+\Big[\sum_{j=1}^\infty\Big(\int_{(B^{*})^c}\chi_{\{\varepsilon_{j+1}
< |y -z | < \varepsilon_{j}\}}(z)|R^H(y,z)|^q  | f(z) - f
_B|^q~dz\Big)^{\rho/q}\Big]^{1/\rho}|x-y|^{1/q'}\\
&\le C\Big[\int_{(B^{*})^c}|R^H(y,z)|^q  | f(z) - f
_B|^q~ dz\Big]^{1/q}|x-y|^{1/q'}\\
&\le
C\left[\sum_{k=1}^\infty\int_{2^kr_0<|x_0-z|<2^{k+1}r_0}\frac{1}{|y-z|^q}|
f(z) - f _B|^q~dz\right]^{1/q}|x-y|^{1/q'}\\
&\le
C\left[\sum_{k=1}^\infty\frac{1}{(2^kr_0)^q}\int_{|x_0-z|<2^{k+1}r_0}|
f(z) - f _B|^q~dz\right]^{1/q}|x-y|^{1/q'}\\
&=
C\left[\sum_{k=1}^\infty\frac{1}{(2^k)^{q-1}}\frac{1}{2^kr_0}\int_{|x_0-z|<2^{k+1}r_0}|
f(z) - f _B|^qdz\right]^{1/q}\left(\frac{|x-y|}{r_0}\right)^{1/q'}\le C\|f\|_{BMO_H(\mathbb{R})}.
\end{align*}
This concludes the proof of Lemma \ref{Lem:geometrico}.
\end{proof}

At this point, as we remarked above, to prove the boundedness of $U$ from $BMO_H(\mathbb{R}^n)$ into $BMO_H(\mathbb{R}^n;E_\rho)$ we need to verify that $U1$ satisfies hypothesis (i) and (ii) of Theorem \ref{criterioT1} in this vector-valued setting.

\vskip0.2cm

\subsubsection{$U1$ satisfies hypothesis (i) of Theorem \ref{criterioT1}.} First note that, by using the properties of the function $\gamma$, it is enough to verify this hypothesis only for the balls $Q_k$ defined by the covering by critical balls, see Section \ref{Section:BMO}. We observe that by \cite[Theorem~A]{CMTT} and \eqref{acoriesz1} we have
\begin{align*}
\lefteqn{\frac{1}{|Q_k|}\int_{Q_k}\|U1(x)\|_{E_\rho}~dx} \\&\le \frac{1}{|Q_k|}\int_{Q_k}\|U(\chi_{Q_k^*})(x)\|_{E_\rho}~dx+\frac{1}{|Q_k|}\int_{Q_k}\int_{(Q_k^*)^c}|R^H(x,y)|~dy~dx\\
&\le\left(\frac{1}{|Q_k|}\int_{Q_k}\|U(\chi_{Q_k^*})(x)\|_{E_\rho}^2dx\right)^{1/2} +\frac{C}{|Q_k|}\int_{Q_k}\int_{(Q_k^*)^c}\frac{e^{-c\left[|x-y|^2+|x||x-y|\right]}}{|x-y|}~dy~dx\\
&\le C\left[1+\frac{1}{|Q_k|}\int_{Q_k}\int_{(Q_k^*)^c} \frac{e^{-c\left[|x-y|^2+|x||x-y|\right]}}{|x-y|}~dy~dx\right]=:C(1+L_k).
\end{align*}
Note that if $x\in Q_k$ and $y\in (Q_k^*)^c$, then $|x-y|\ge \gamma(x_k)$. We now distinguish two cases. If $|x_k|\le 1$ then $\gamma(x_k)=1/2$ and we can write
$$L_k\le\frac{1}{|Q_k|}\int_{Q_k}\int_{(Q_k^*)^c}\frac{e^{-c|x-y|^2}}{|x-y|}~dy~dx\le 2\int_{-\infty}^\infty e^{-cu^2}du.$$
On the other hand, if $|x|\ge 1$, $\gamma(x)\sim \frac{1}{|x|}$. Moreover, $\gamma(x)\sim\gamma(x_k)$ provided that $x\in Q_k$. Hence, if $|x_k|\ge 1$ we get
$$L_k\le \frac{1}{|Q_k|}\int_{Q_k}\int_{|x-y|\ge\gamma(x_k)}\frac{e^{-c|x-y|/\gamma(x_k)}}{|x-y|}~dy~dx\le\frac{C}{|Q_k|}\int_{Q_k}\int_{|x-y|\ge
\gamma(x_k)}\frac{\gamma(x_k)}{|x-y|^2}~dy~dx=C.$$
Therefore, hypothesis (i) in Theorem \ref{criterioT1} holds.

\subsubsection{$U1$ satisfies hypothesis (ii) of Theorem \ref{criterioT1}.} Let $B=B(x_0,r_0)$, where $x_0\in \mathbb{R}$ and   $0<r_0<\gamma(x_0)$. Let
$$R^{(k)}(x,y):=\frac{x-y}{4\pi}\int_0^{\gamma(x_k)^2}\left(\log\frac{1+s}{1-s}\right)^{-1/2} \frac{e^{-\frac{|x-y|^2}{4s}}}{(s(1-s^2))^{1/2}}~\frac{ds}{s},\qquad x,y\in\mathbb R.$$
It is clear that for every $0<\varepsilon<\eta<\infty$, $\displaystyle\int_{\varepsilon<|x-y|<\eta}R^{(k)}(x,y)~dy=0.$  Let $x,y \in B $, then
\begin{align}\label{F1}
\lefteqn{\|U1(x)-U1(y)\|_{E_\rho}}\nonumber \\
&=\sup_{\varepsilon_j}\left[\sum_{j=1}^\infty
\Big|\int_{\varepsilon_{j+1}<|x-z|<\varepsilon_j}R^H(x,z)~dz- \int_{\varepsilon_{j+1}<|y-z|<\varepsilon_j}R^H(y,z)~dz\Big|^\rho\right]^{1/\rho}\nonumber\\
&=\sup_{\varepsilon_j}\left[\sum_{j=1}^\infty
\Big|\int_{\varepsilon_{j+1}<|x-z|<\varepsilon_j}(R^H(x,z)-R^{(k)}(x,z))~dz\right. \nonumber\\
&\qquad\qquad\left.- \int_{\varepsilon_{j+1}<|y-z|<\varepsilon_j} (R^H(y,z)-R^{(k)}(y,z))~dz\Big|^\rho\right]^{1/\rho}\nonumber\\
&\le\sup_{\varepsilon_j}\left[\sum_{j=1}^\infty
\Big|\int_{\varepsilon_{j+1}<|x-z|<\varepsilon_j}(R^H(x,z)-R^{(k)}(x,z)-R^H(y,z) +R^{(k)}(y,z))~dz\Big|^\rho\right]^{1/\rho}\nonumber\\
&\quad+\sup_{\varepsilon_j}\left[\sum_{j=1}^\infty \Big|\int_{-\infty}^\infty
(\chi_{\{\varepsilon_{j+1}<|y-z|<\varepsilon_j\}}(z)- \chi_{\{\varepsilon_{j+1}<|x-z|<\varepsilon_j\}}(z))(R^H(y,z)-R^{(k)}(y,z))~dz\Big|^\rho\right]^{1/\rho}\nonumber\\
&\le \int_{-\infty}^\infty
\Big|R^H(x,z)-R^{(k)}(x,z)-R^H(y,z)+R^{(k)}(y,z)\Big|~dz\nonumber\\
&\quad+\sup_{\varepsilon_j}\left[\sum_{j=1}^\infty \Big|\int_{-\infty}^\infty
(\chi_{\{\varepsilon_{j+1}<|y-z|<\varepsilon_j\}}(z)-\chi_{\{\varepsilon_{j+1}<|x-z|<\varepsilon_j\}}(z)) (R^H(y,z)-R^{(k)}(y,z))~dz\Big|^\rho\right]^{1/\rho}\nonumber\\
&=:Z_1(x,y)+\sup_{\varepsilon_j\searrow 0}Z_{2,\varepsilon_j}(x,y),\qquad x,y\in B.
\end{align}

\noindent\textit{Estimate of $Z_1$}. The difference involving $R^H-R^{(k)}$ in $Z_1$ is decomposed, up to the multiplicative constant $(2\pi)^{-1}$, as follows
\begin{align}\label{F2}
\lefteqn{R^H(x,z)-R^{(k)}(x,z)-(R^H(y,z)-R^{(k)}(y,z))=\int_y^x\frac{\partial}{\partial
u}(R^H(u,z)-R^{(k)}(u,z))~du}\nonumber\\
&=\int_y^x\left\{\int_0^1\left(\log\frac{1+s}{1-s}\right)^{-1/2} \frac{e^{-\frac{1}{4}\left[s(u+z)^2+\frac{1}{s}(u-z)^2\right]}}{(s(1-s^2))^{1/2}}\left[\frac{1}{2}\left(\frac{1}{s}+s\right) -\frac{1}{4}\left(s(u+z)+\frac{u-z}{s}\right)^2\right]ds\right.\nonumber\\
&\qquad\left.-
\int_0^{\gamma(x_k)^2}\left(\log\frac{1+s}{1-s}\right)^{-1/2}\frac{e^{-\frac{|u-z|^2}{4s}}}{(s(1-s^2))^{1/2}} \left[\frac{1}{2s}-\frac{1}{4}\left(\frac{u-z}{s}\right)^2\right]ds\right\}du\nonumber\\
&=\int_y^x\Big\{\int_0^1\left(\log\frac{1+s}{1-s}\right)^{-1/2} \frac{e^{-\frac{1}{4}\left[s(u+z)^2+\frac{1}{s}(u-z)^2\right]}}{(s(1-s^2))^{1/2}}\left[-s+ \frac{s^2}{2}(u+z)^2+(u+z)(u-z)\right]ds\nonumber\\
&\qquad+\int_{\gamma(x_k)^2}^1\left(\log\frac{1+s}{1-s}\right)^{-1/2} \frac{e^{-\frac{1}{4}\left[(s(u+z)^2+\frac{1}{s}(u-z)^2\right]}}{(s(1-s^2))^{1/2}} \left[-\frac{1}{s}+\frac{1}{2}\left(\frac{u-z}{s}\right)^2\right]ds\nonumber\\
&\qquad+\int_0^{\gamma(x_k)^2}\left(\log\frac{1+s}{1-s}\right)^{-1/2}\frac{1}{(s(1-s^2))^{1/2}} \left[-\frac{2}{s}+\frac{1}{2}\left(\frac{u-z}{s}\right)^2\right]\nonumber\\
&\qquad\qquad\qquad\qquad\times2\left[e^{-\frac{1}{4}\left[s(u+z)^2+\frac{1}{s}(u-z)^2\right]}
-e^{-\frac{|u-z|^2}{4s}}\right]~ds\Big\}~du\nonumber\\
&=:I_1(x,y,z)+I_2(x,y,z)+I_3(x,y,z),\,\,\,x,y\in
Q_k^*,\,\,y<x\,\,\,\hbox{and}\,\,\,z\in \mathbb{R}.
\end{align}

For $I_1$ we have that
\begin{align}\label{F3}
\int_{-\infty}^\infty |I_1(x,y,z)|~dz&\le C\int_{-\infty}^\infty
\int_y^x\int_0^1\left(\log\frac{1+s}{1-s}\right)^{-1/2} \frac{e^{-\frac{1}{8}\left[s(u+z)^2+\frac{1}{s}(u-z)^2\right]}}{(s(1-s))^{1/2}}~ds~du~dz\nonumber\\
&\le C\int_y^x\int_0^1\left(\log\frac{1+s}{1-s}\right)^{-1/2}\frac{1}{(s(1-s))^{1/2}}\int_{-\infty}^\infty
e^{-\frac{|u-z|^2}{8s}}dz~ds~du\nonumber\\
&\le C\int_y^x\int_0^1\left(\log\frac{1+s}{1-s}\right)^{-1/2}\frac{1}{(1-s)^{1/2}}~ds~du\nonumber\\
&\le C|x-y|,\qquad x,y\in Q_k^*,~y<x.
\end{align}

For $I_2$,
\begin{align}\label{F4}
\int_{-\infty}^\infty |I_2(x,y,z)|~dz &\le \int_{-\infty}^\infty
\int_y^x\left[\int_{\gamma(x_k)^2}^{1/2}+\int_{1/2}^1\right]\left(\log\frac{1+s}{1-s}\right)^{-1/2} \frac{e^{-\frac{|u-z|^2}{8s}}}{s^{3/2}(1-s)^{1/2}}~ds~du~dz\nonumber\\
&\le C\int_y^x\left[1+\int_{\gamma(x_k)^2}^{1/2}\frac{1}{s^2}\int_{-\infty}^\infty e^{-\frac{t^2}{8s}}~dt~dz\right]~du\nonumber\\
&\le C\int_y^x\left[1+\int_{\gamma(x_k)^2}^{1/2}\frac{1}{s^{3/2}}~ds\right]~du\le\frac{|x-y|}{\gamma(x_k)},\qquad x,y\in Q_k^*,~y<x.
\end{align}

For $I_3$,
\begin{align*}
\lefteqn{\int_{-\infty}^\infty |I_3(x,y,z)|~dz}\\
&\le C\int_{-\infty}^\infty
\int_y^x\int_0^{\gamma(x_k)^2}\frac{1}{s}\Big|-\frac{1}{2s}+\frac{1}{4}\Big(\frac{u-z}{s}\Big)^2 \Big|\Big|e^{-\frac{1}{4}s(u+z)^2}-1\Big|e^{-\frac{1}{4}\frac{(u-z)^2}{s}}~ds~du~dz\\
&\le C\int_{-\infty}^\infty
\int_y^x\int_0^{\gamma(x_k)^2}\frac{1}{s}(u+z)^2e^{-\frac{1}{8}\frac{(u-z)^2}{s}}~ds~du~dz\\
&\le C\int_y^x\int_0^{\gamma(x_k)^2}\frac{1}{s}\int_{0}^\infty\left(u^2+t^2\right)e^{-\frac{t^2}{16s}}~dt~ds~du\\
&\le C\int_y^x\int_0^{\gamma(x_k)^2}\frac{1+u^2}{s^{1/2}}~ds~du\le C\gamma(x_k)\left(|x^3-y^3|+|x-y|\right),\qquad x,y\in Q_k^*,~y<x.
\end{align*}
Since $|x^3-y^3|\le |x-y|(x^2+y^2+|xy|)$ and
$\gamma(a)\sim\gamma(x_k)$, $a\in Q_k^*$, it follows that
$|x^3-y^3|\le C|x-y|/\gamma(x_k)^2$, $x,y\in Q_k^*$. Hence
\begin{equation}\label{F5}
\int_{-\infty}^\infty |I_3(x,y,z)|~dz\le
C\frac{|x-y|}{\gamma(x_k)},\,\,\,x,y\in Q_k^*.
\end{equation}
We conclude, by combining \eqref{F2}--\eqref{F5}, that
\begin{equation}\label{F6}
Z_1(x,y)\le C\frac{|x-y|}{\gamma(x_k)},\,\,\,x,y\in Q_k^*.
\end{equation}

\noindent\textit{Estimate of $Z_{2,\varepsilon_j}$}. To this end, let us decompose
$R^H(y,z)-R^{(k)}(y,z)$, up to the multiplicative constant $(4\pi)^{-1}$, as follows
\begin{align}\label{F7}
\lefteqn{R^H(y,z)-R^{(k)}(y,z)=\int_0^1\left(\frac{s}{1-s^2}\log\frac{1+s}{1-s}\right)^{-1/2} \frac{s(y+z)}{1-s^2}~e^{-\frac{1}{4}\left[s(y+z)^2+\frac{1}{s}(y-z)^2\right]}~ds}\nonumber\\
&\quad+\int_{\gamma(x_k)^2}^1\left(\frac{s}{1-s^2}\log\frac{1+s}{1-s}\right)^{-1/2} \frac{y-z}{s(1-s^2)}~e^{-\frac{1}{4}\left[s(y+z)^2+\frac{1}{s}(y-z)^2\right]}~ds\nonumber\\
&\quad+\int_0^{\gamma(x_k)^2}\left(\frac{s}{1-s^2}\log\frac{1+s}{1-s}\right)^{-1/2}\frac{y-z}{s(1-s^2)} \left[e^{-\frac{1}{4}\left[s(y+z)^2+\frac{1}{s}(y-z)^2\right]} -e^{-(y-z)^2/4s}\right]~ds\nonumber\\
&=:J_1(y,z)+J_2(y,z)+J_3(y,z),\qquad y,z\in \mathbb{R}.
\end{align}
By proceeding as in the geometric analysis of \eqref{sustituido} we can obtain
\begin{align*}
\lefteqn{\Big|\int_{-\infty}^\infty(\chi_{\varepsilon_{j+1}<|x-z|<\varepsilon_j}(z)- \chi_{\varepsilon_{j+1}<|y-z|<\varepsilon_j}(z)) (R^H(y,z)-R^{(k)}(y,z))~dz\Big|}\\
&\le C\left[\int_{-\infty}^\infty
\chi_{\varepsilon_{j+1}<|y-z|<\varepsilon_j}(z)\Big|R^H(y,z)-R^{(k)}(y,z)\Big|^q~dz\right]^{1/q} |x-y|^{1/q'}\\
&\quad+C\left[\int_{-\infty}^\infty
\chi_{\varepsilon_{j+1}<|x-z|<\varepsilon_j}(z)\Big|R^H(y,z)-R^{(k)}(y,z)\Big|^q~dz\right]^{1/q}|x-y|^{1/q'},\quad x,y\in \mathbb{R},
\end{align*}
where $1<q<\infty$. We take $1<q<\rho$. Then, by \eqref{F7},
\begin{align}\label{F8}
Z_{2,\varepsilon_j}(x,y) &\le C\left[\left(\sum_{j=1}^\infty\int_{-\infty}^\infty
\chi_{\varepsilon_{j+1}<|y-z|<\varepsilon_j}(z)\Big|R^H(y,z)-R^{(k)}(y,z)\Big|^q~dz\right)^{\rho/q}\right]^{1/\rho} |x-y|^{1/q'} \nonumber\\
&\quad+C\left[\sum_{j=1}^\infty\left(\int_{-\infty}^\infty
\chi_{\varepsilon_{j+1}<|x-z|<\varepsilon_j}(z) \Big|R^H(y,z)-R^{(k)}(y,z)\Big|^q~dz\right)^{\rho/q}\right]^{1/\rho}|x-y|^{1/q'}\nonumber\\
&\le C\left(\int_{-\infty}^\infty\Big|R^H(y,z)-R^{(k)}(y,z)\Big|^q~dz\right)^{1/q}|x-y|^{1/q'}\nonumber\\
&\le C\sum_{j=1}^3\left[\int_{-\infty}^\infty|J_j(y,z)|^q~dz\right]^{1/q}|x-y|^{1/q'},\quad x,y\in \mathbb{R}.
\end{align}

For $J_1$, Minkowski's inequality implies that
\begin{equation}\label{F9}
\left[\int_{-\infty}^\infty|J_1(y,z)|^q~dz\right]^{1/q}\le
C\int_0^1\left((1-s)\log\frac{1+s}{1-s}\right)^{-1/2}\left[\int_{-\infty}^\infty
e^{-\frac{q|y-z|^2}{8s}}~dz\right]^{1/q}ds\le C,
\end{equation}
for $y\in\mathbb R$.

For $J_2$,
\begin{align}\label{F10}
\left[\int_{-\infty}^\infty|J_2(y,z)|^qdz\right]^{1/q}&\le C\int_{\gamma(x_k)^2}^1\left(\log\frac{1+s}{1-s}\right)^{-1/2}\frac{1}{s(1-s)^{1/2}}\left[\int_{-\infty}^\infty e^{-\frac{q|y-z|^2}{8s}}dz\right]^{1/q}ds\nonumber\\
&\le C\int_{\gamma(x_k)^2}^1\left(\log\frac{1+s}{1-s}\right)^{-1/2}\frac{s^{-1+1/2q}}{(1-s)^{1/2}}~ds\nonumber\\
&\le C\gamma(x_k)^{-1/q'},\qquad y\in \mathbb{R}.
\end{align}

For $J_3$ we have that
\begin{align}\label{F11}
|J_3(y,z)|&\le C\int_0^{\gamma(x_k)^2}\frac{|y-z|}{s^2}\Big|e^{-\frac{s(y+z)^2}{4}}-1\Big|e^{-\frac{(y-z)^2}{4s}}~ds\nonumber\\
&\le C(y+z)^2\int_0^{\gamma(x_k)^2}\frac{e^{-\frac{(y-z)^2}{8s}}}{s^{1/2}}~ds,\qquad y,z\in \mathbb{R}.
\end{align}
Then, since $\gamma(y)\sim \gamma(x_k)$ provided that $y\in Q_k^*$,
\begin{align}\label{F13}
\left[\int_{-\infty}^\infty J_3(y,z)^q~dz\right]^{1/q}&\le C\int_0^{\gamma(x_k)^2}\frac{1}{s^{1/2}}\left[\int_{-\infty}^\infty \left(y^{2q}+t^{2q}\right)e^{-qt^2/8s}dt\right]^{1/q}ds\nonumber\\
&\le Cy^2\int_0^{\gamma(x_k)^2}s^{(1-q)/2q}~ds+C\int_0^{\gamma(x_k)^2}s^{(1+q)/2q}~ds\nonumber\\
&\le Cy^2\gamma(x_k)^{(q+1)/q}+C\gamma(x_k)^{(3q+1)/q}\nonumber\\
&\le C\frac{\gamma(x_k)^{(q+1)/q}}{\gamma(y)^2}+C\gamma(x_k)^{(3q+1)/q}\le C\gamma(x_k)^{-1/q'},\,\,\,y\in Q_k^*.
\end{align}

By combining now \eqref{F8}--\eqref{F13} we conclude that, for every
$x,y\in B$,
\begin{equation}\label{F14}
\sup_{\varepsilon_j}Z_{2,\varepsilon}(x,y)\le
C\left(\frac{|x-y|}{\gamma(x_k)}\right)^{1/q'}.
\end{equation}
This finishes the estimate of $Z_{2,\varepsilon_j}$.

\vskip0.2cm

From \eqref{F1}, \eqref{F6} and \eqref{F14} we deduce that
$$\|U1(x)-U1(y)\|_{E_\rho}\le C\left[\frac{r_0}{\gamma(x_k)}+\left(\frac{r_0}{\gamma(x_k)}\right)^{1/q'}\right],$$
for every $x,y\in B$. Then, by \cite[Lemma~2.2]{DGMTZ}, since $\gamma(x_k)\sim \gamma(x_0)$, it follows that
\begin{align*}
\frac{1}{|B|^2}\int_B\int_B\|U1(x)-U1(y)\|_{E_\rho}~dy~dx\le C\left[\frac{r_0}{\gamma(x_k)}+ \left(\frac{r_0}{\gamma(x_k)}\right)^{1/3}\right].
\end{align*}
We leave to the reader to check that this is stronger than hypothesis (ii) in Theorem \ref{criterioT1}. Hence the proof of Theorem \ref{variacion} is finished.

\end{document}